\documentclass{amsart}

\usepackage[T1]{fontenc}
\usepackage{latexsym,amsfonts}
\usepackage{tikz}
\usetikzlibrary{calc,shapes,arrows,automata}
\usepackage{amsmath}
\usepackage{algorithm}
\usepackage{algorithmic}
\usepackage{url}
\usepackage{subcaption} 
\usepackage{caption}
\usepackage{centernot}
\usepackage{svg}
\usepackage{epsfig}
\usepackage{thmtools}
\usepackage{hyperref}
\usepackage{multirow}
\usepackage{preamble}

\begin{document}

\begin{abstract}
    \label{sec:abstract}
    In this paper, we introduce two new types of barrier certificates that are based on multiple functions rather than a single one. 
    A conventional barrier certificate for a stochastic dynamical system is a nonnegative real-valued function whose expected value does not increase as the system evolves. 
    This requirement guarantees that the barrier certificate forms a nonnegative supermartingale and can be used to derive a lower bound on the probability that the system remains safe.
    A key advantage of such certificates is that they can be automatically searched for using tools such as optimization programs instantiated with a fixed template.
    When this search is unsuccessful, the common practice is to modify the template and attempt the synthesis again.
    Drawing inspiration from logical interpolation, we first propose an alternative framework that uses a collection of functions to jointly serve as a barrier certificate. 
    We refer to this construct as an interpolation-inspired barrier certificate.
    Nonetheless, we observe that these certificates still require one function in the collection to satisfy a supermartingale condition.
    Motivated by recent work in the literature, we next combine $k$-induction with interpolation-inspired certificates to relax this supermartingale constraint.
    We develop a general and more flexible notion of barrier certificates, which we call $k$-inductive interpolation-inspired barrier certificates. 
    This formulation encompasses multiple ways of integrating interpolation-inspired barrier certificates with $k$-induction. 
    We highlight two specific instantiations among these possible combinations. 
    For polynomial systems, we employ sum-of-squares (SOS) programming to synthesize the corresponding set of functions. 
    Finally, through our case studies, we show that the proposed methods enable the use of simpler templates and yield tighter lower bounds on the safety probability.
\end{abstract}

\title[$k$-Inductive and Interpolation-Inspired Barrier Certificates for Stochastic Dynamical Systems]{$k$-Inductive and Interpolation-Inspired Barrier Certificates for Stochastic Dynamical Systems}

\thanks{
This work was supported by NSF under grants CNS-2111688 and CNS-2145184.}

\author[mohammed adib oumer]{Mohammed Adib Oumer} 
\author[vishnu murali]{Vishnu Murali} 
\author[majid zamani]{Majid Zamani}

\address{Department of Computer Science at the University of Colorado, Boulder, CO, USA.}
\email{\{mohammed.oumer,~vishnu.murali,~majid.zamani\}@colorado.edu}
\urladdr{https://www.hyconsys.com/members/moumer/}
\urladdr{https://www.hyconsys.com/members/vmurali/}
\urladdr{https://www.hyconsys.com/members/mzamani/}

\maketitle

\section{Introduction} 
\label{sec:intro}

Barrier certificates are a prominent method to verify the safety of dynamical systems. The authors of \cite{prajna2004safety}, propose the notion of a barrier certificate for continuous-time systems as a real-valued function whose zero level set separates the unsafe region from all the reachable region. The value of the barrier certificate is positive over a given set of unsafe states, nonpositive over a given set of initial states, and nonincreasing as a system evolves.
Thus, it acts as an inductive proof of safety.
This certificate is also used for stochastic systems \cite{prajna2004stochastic, prajna2007framework} where it provides probabilistic lower bounds on safety guarantees.
Here, a barrier certificate is a nonnegative real-valued function, whose value is greater than or equal to $1$ for the unsafe states, less than $1$ for the initial states and it remains nonincreasing in \emph{expectation} along the trajectory of the system. 
Thus, the certificate acts as a supermartingale as a system evolves and provides lower bounds on the probability of the system being safe. 
Although the search for such a certificate is automatable, it relies on users fixing a template.
The search for barrier certificates is carried out using SMT-based approaches \cite{de2011satisfiability} or optimization \cite{parrilo2003semidefinite,prajna2004safety}. In both of these approaches, we search for a function in a fixed template satisfying the conditions characterizing barrier certificates. Typically, when such a function is not successfully found for a given template, a different template is considered.
We instead address this challenge by proposing generalized notions of barrier certificates inspired by interpolation \cite{mcmillan2003interpolation} and $k$-induction \cite{sheeran2000checking, donaldson2011software}, which are typically used in establishing inductive invariants for hardware and software systems.

An inductive invariant is a property defined over the states of a system such that: $i)$ it holds for all initial states, and $ii)$ whenever it holds for a given state, it also holds for the successor state obtained by applying the transition function. 
By induction, this guarantees that the inductive invariant holds for every reachable state of the system.
If, in addition, the negation of this property holds for all unsafe states, then it constitutes a proof of safety. 
However, a property that is valid for all reachable states is often not inductive as is.
In these situations, a typical strategy is to \emph{strengthen} the property—using techniques such as interpolation or alternative inductive frameworks like $k$-induction—by expressing it as a conjunction of several properties, thereby obtaining an inductive invariant.

A barrier certificate is a discretization-free functional inductive invariant and, as such, the search for such functions suffers from similar issues.
Typically, the conditions of barrier certificates are imposed over a single function.
However, these conditions can be restrictive.
One option to change the notion of induction, as shown in \cite{anand2022kstochastic}, is through the use of $k$-induction. 
Here, the authors still rely on a single function to act as a $k$-inductive barrier certificate, where they relax the supermartingale condition at each time step to a $c$-martingale condition for less than $k$ steps and a supermartingale condition for every $k$ steps. 

Another option to relax the standard conditions is by allowing multiple functions to act as barrier certificates via interpolation \cite{mcmillan2003interpolation}, which we employ in this work. 
We consider a notion of interpolation-inspired barrier certificates that allows us to use a broader range of function templates as proofs of safety, and can be searched for similar to standard barrier certificates. 
Next, we generalize the notion of $k$-inductive barrier certificates proposed in \cite{anand2022kstochastic} by considering a $k$-inductive barrier certificate as a set of functions.
Finally, as both $k$-induction and interpolation provide valid but incomparable advantages when trying to find barrier certificates, we show that they can be combined to form a more general formulation of barrier certificates, which we call $k$-inductive interpolation-inspired barrier certificates.

\vspace{0.5em}
\noindent \textbf{Contributions.}
The contributions of the paper are listed below.

\begin{enumerate}
    \item We introduce new concepts of interpolation-inspired barrier certificates and $k$-inductive interpolation-inspired barrier certificates, whose conditions are less restrictive than those of standard barrier certificates.
    \item We formulate an objective-driven SOS programming approach to search for the proposed barrier certificates using a fixed polynomial template.
    \item We show that these new certificates can admit simpler templates as valid certificates, thereby making the search for them computationally more tractable.
    \item We illustrate that the proposed certificates enhance the safety probability of a stochastic system.
    \item We further demonstrate that, by appropriately choosing certain hyperparameters, our relaxed formulations recover the classical barrier-certificate conditions, establishing that our framework strictly generalizes the standard setting.
\end{enumerate}

\vspace{0.5em} \noindent \textbf{Related works.}
Inductive invariants and incremental inductive proofs are important tools in verifying the safety of finite state-transition systems as seen in \cite{zhang2004incremental, cabodi2008strengthening, bradley2011sat}. 
Typically, such systems are described as a set of logical variables, where the initial and unsafe states, as well as the transition relation, are described as logical formulae.
The safety verification goal is to ensure that the negation of the formula representing the unsafe states is an inductive invariant.
Unfortunately, this is often not inductive.
Thus, a prominent approach to make it inductive is to incrementally strengthen this formula via 
interpolation-based approaches \cite{mcmillan2003interpolation, bradley2012understanding}. 
Given a property of interest, such proofs check whether it is an inductive invariant. 
If not, they try to incrementally constrain it till an inductive proof is obtained. 
In the context of bounded model checking, (logical) interpolation unrolls the transition function a fixed number of times and finds intermediate formulae called interpolants until an inductive invariant formula is found.
IC3 \cite{bradley2011sat,bradley2012understanding} uses overapproximating frames and counterexamples to build incremental formulae one step at a time until an inductive invariant is found. Both of these approaches use multiple formulae to form an inductive invariant. 

Since barrier certificates are functional inductive invariants, we now compare our work with works that used multiple functions as barrier certificates. 
The design and utilization of standard barrier certificates for stochastic systems has been carried out in various works such as \cite{prajna2004stochastic,jagtap_2020_formal,huang2017probabilistic,clark2021control,xue2024reach,xue2024sufficient}.
The authors of~\cite{anand2019verification, feng2020unbounded} consider multiple functions as safety certificates for stochastic systems. 
Our work differs from the latter as follows.
The authors of \cite{anand2019verification} consider multiple barrier certificates for the case of switched stochastic systems. They design different standard barrier certificates for each mode of a switched system. In contrast, we consider multiple functions that replace the purpose of a standard barrier certificate. For instance, an interpolation-inspired barrier certificate can be designed for each mode of a switched system. 
The authors of \cite{feng2020unbounded} propose a similar idea as \cite{sogokon2018vector} for continuous time stochastic dynamical systems. The similarity with our work is in the use of multiple functions and exploiting supermartingale properties for bounding the probabilities. The first difference is that the conditions for the initial and unsafe states are imposed over all the functions in their work while our proposed method imposes the conditions over all the functions only for the unsafe states. Furthermore, in their study, they assign the relationship between the functions to a Metzler matrix, whereas our approach does not have to be represented in this manner.
Instead of all the functions, we only need one function to act as a supermartingale as the system evolves.
The authors of \cite{lewis2024verification} considered a notion of time-varying $k$-inductive barrier certificates.
We use a similar argument to propose a $k$-inductive barrier certificate that uses multiple functions. These functions can be considered time-varying such that we pick the $i^{th}$ function for every $t = rk + i$ time-step for some nonnegative integer $r$.
We refer the interested readers to Section~\ref{ss:relaxk} for more details.

\vspace{0.5em} \noindent \textbf{Organization.}
In Section \ref{sec:prelims}, we outline notations, define stochastic dynamical systems, and review standard barrier certificates.
In Section \ref{sec:interp-k}, we introduce the first key theoretical result of our paper by proposing interpolation-inspired barrier certificates.
We then generalize $k$-inductive barrier certificates and present the second key theoretical result: combining them with interpolation-inspired barrier certificates.
An implementation of the proposed techniques is discussed in Section \ref{sec:synthesis}, followed by case studies
in Section \ref{sec:case}.

\section{Preliminaries}
\label{sec:prelims}

We consider the probability space $(\Omega, \Ff, \Pro)$ where $\Omega$ is the sample space, $\Ff$ is the $\sigma$-algebra on $\Omega$ that contains subsets of $\Omega$ as events in the probability space, and $\Pro$ is the probability measure that assigns to each event in the event space a probability, which is a number between 0 and 1. We consider random variables to be measurable functions of the form $X: (\Omega, \Ff) \rightarrow (S_X, \Ff_X)$ from the sample space $\Omega$ to another measurable space $S$ called the state space. Each random variable $X$ is associated with a probability measure on $(S_X, \Ff_X)$ as $\Pr \{Z\} = \Pro \{X^{-1}(Z)\}$ for any $Z \in \Ff_X$.

Let $X$ be topological space. 
The collection of all Borel sets on $X$ forms the Borel $\sigma$-algebra $B(X)$. 
The map $f: X\rightarrow Y$ is said to be measurable when it is Borel-measurable.

We use $\N$ and $\R$ to denote the set of natural numbers and reals, respectively. 
For $m \in \R$, we use $\R_{\geq m}$ and $\R_{> m}$ to denote the intervals $[m, \infty)$ and $(m,\infty)$, respectively.
Similarly, for any natural number $n \in \N$, we use $\N_{\geq n}$ to denote the set of natural numbers greater than or equal to $n$. The $n$-dimensional Euclidean space is denoted by $\mathbb{R}^n$.

We use $\exists$ and $\forall$ to denote the existential and universal quantifiers, respectively. We use logical operators $\lor,\ \land,\ \neg$ and $\implies$ for disjunction (logical OR), conjunction (logical AND), negation (logical NOT) and implication, respectively.

For a function $f: \Xx \times \mathcal{A} \rightarrow \Xx$ and $k \in \N_{\geq 1}$, we use $f^k: \Xx \times \mathcal{A}^k \to \Xx$ to denote the composition of the function $f$ by itself $k$-times (i.e. given a set of $k$ values $ (a_0, a_1, \dots ,a_{k-1})$, we define $f^k(x,a_0) = f(x,a_0)$ for $k=1$ and $f^k(x, (a_0, a_1, \dots ,a_{k-1})) = f(f^{k-1}(x,(a_0, \dots ,a_{k-2})),a_{k-1})$ for $k>1$).
A set of $N+1$ functions is denoted using $\Bb_i,\ \forall i \in \{0,1,\dots,N\}$.
Given a collection of sets $\Xx_i,\ i = \{0,1,\dots,N\}$, we use $\bigcup_{i=0}^N \Xx_i$ to denote the union of the sets $\Xx_i$. Given two sets $\Xx$ and $\Yy$, $\Xx \backslash \Yy := \{x: x\in \Xx \text{ and } x\notin \Yy\}$.
We use $\Nn(\mu,\sigma^2)$ to denote a normal distribution with mean $\mu \in \R$ and variance $\sigma^2 \in \R_{>0}$.

\subsection{Stochastic Dynamical System}
\label{ss:prelims_system}

\begin{definition}
    A discrete-time stochastic dynamical system (dt-SS) $\Sys$ is given by the tuple:
    \begin{align}
        \Sys = (\Xx,\Xx_0, w, f),
    \end{align} where
    \begin{itemize}
        \item $\Xx \subseteq \R^n$ is a Borel space that represents the state set of $\Sys$;
        \item $\Xx_0 \subseteq X$ denotes a set of initial states;
        \item $w:= \{w(t): \Omega \rightarrow \Ww,\ t \in \N \}$ is a sequence of i.i.d random variables from a sample space $\Omega$ to the measurable space $(\Ww, \Ff_w)$, commonly interpreted as system noise; and
        \item $f: \Xx \times \Ww \rightarrow \Xx$ is a measurable function that describes the state evolution of $\Sys$.
    \end{itemize}
    \label{def:system}
\end{definition} 
For $x(t)$, the state of the system at time step $t\in \mathbb{N}$, the state of the system in the next time step is given by the following stochastic difference equation:
\begin{align}
    x(t+1) = f(x(t), w(t)), && \forall x(t) \in \Xx.
\end{align}

We use $\mathbf{x}_{x_0} = (x(0),x(1),x(2),\dots)$ to denote the solution process generated by $\Sys$ starting from the initial state $x(0) = x_0 \in \Xx_0$. 
We denote the state at time step $t \in \N$ for the solution process $\mathbf{x}_{x_0}$ as $\mathbf{x}_{x_0}(t)$ .
Now we define reachable states of a dt-SS.
\begin{definition}[Reachability]
  We say a state $x(t_1)$ of a dt-SS $\Sys$ given in Definition \ref{def:system} is reachable from the state $x(t_0)$ if there exists a solution process $\mathbf{x}_{x(t_0)}$ which contains $x(t_1)$. That is, $x(t_1) = f^i(x(t_0),w_i(t_0))$, for some $i \in \mathbb{N}$ and some $w_i(t_0) = [w(t_0);\dots;w(t_0+i-1) = w(t_1-1)]$, which is a vector of noise terms from $t_0$ to $t_1-1$.
  \label{def:reachability}
\end{definition}

Since the codomain of the map $f$ is $\Xx$, this implicitly implies that the state set $\Xx$ is forward invariant, which might seem conservative when dealing with unbounded noise, especially when $\Xx$ is bounded. Following the convention introduced in \cite{kushner1967stochastic, xue2024sufficient, anand2022kstochastic}, to ensure the forward invariance of $\Xx$, we adopt the standard assumption of stopping the stochastic process if it reaches the boundary of $\Xx$: 
\begin{assumption}
\label{assumption:stoppedprocess}
For any solution process $\mathbf{x}_{x_0}$
of dt-SS $\Sys$ starting
from some initial state $x_0 \in \Xx_0$, we have $\mathbf{x}_{x_0}(t) \in \Xx$ for all $t\in \N$. This is ensured by considering a ``stopped process" $\overline{\mathbf{x}}_{x_0}(t)$ given as:
\begin{align}
    \overline{\mathbf{x}}_{x_0}(t) = 
    \begin{cases}
        \mathbf{x}_{x_0}(t) & \forall~t<\tau,\\
        \mathbf{x}_{x_0}(\tau-1) & \forall~t\geq\tau.
    \end{cases}
\end{align}
where $\tau\in\N$ is the first exit time of $\mathbf{x}_{x_0}$ from $\Xx$.
\end{assumption}

Our understanding of a stopped process is consistent with the treatment in~\cite{abate2025quantitative,henzinger2025supermartingale}, where the transition function $f$ is interpreted as piecewise-defined: it follows the system dynamics while the state remains in $\Xx$, but not once it leaves $\Xx$ (that is, $f(x,w) = x$ whenever $f(x,w) \notin \Xx$).
From this point on, we assume for simplicity of exposition that the state space $\Xx$ is forward invariant, and we no longer explicitly refer to the stopped process.

As we deal with stochastic systems with noise which consist of unbounded support, we are interested in obtaining probabilistic guarantees over the satisfaction of safety for a given dt-SS $\Sys$.
Particularly, we would like to compute a tight lower bound on the probability of satisfying safety. We present the definition of probabilistic satisfaction of safety below.

\begin{definition}[Safety Probability]
\label{def:safetyprob}
    Let  $\Xx_0 \subseteq \Xx$ and $\Xx_u \subseteq \Xx$ represent the set of initial states and unsafe states, respectively for a dt-SS $\Sys$ given in Definition \ref{def:system}. We say $\Sys$ satisfies safety with a probability bound of $\lambda$ if the solution processes of $\Sys$ starting from any $x_0 \in \Xx_0$ never reach $\Xx_u$ with a probability of at least $p$, i.e.
    \begin{align*}
        \Pro\{\mathbf{x}_{x_0}(t) \notin \Xx_u,\ \forall t \in \N\} \geq p,~~~\forall x_0\in \Xx_0.
    \end{align*}
\end{definition} 
The goal of safety verification for a dt-SS $\Sys$ is to compute the probability bound constant $0\leq p \leq 1$.

\subsection{Barrier Certificates}
\label{ss:prelims_bc}

For safety verification of a dt-SS $\mathcal{S}$ as in Definition \ref{def:system}, we now discuss the notion of barrier certificates \cite{prajna2004safety} that provide sufficient conditions for safety.

\begin{definition}[Barrier Certificate]
\label{def:bc}
   A function $\Bb: \Xx \to \R$ is a barrier certificate for a dt-SS $\Sys$ with respect to a set of initial states $\Xx_0$ and a set of unsafe states $\Xx_u$ if there exists a constant $0\leq \gamma \leq 1$ such that: 
    \begin{align}
        \label{eq:bar_cond_0}
        & \Bb(x) \geq 0 && \forall x \in \Xx,\\
        \label{eq:bar_cond_1}
        & \Bb(x) \leq \gamma && \forall x \in \Xx_0, \\
        \label{eq:bar_cond_2}
        & \Bb(x) \geq 1 && \forall x \in \Xx_u, \\  
        \label{eq:bar_cond_3}
        & \E[\Bb(f(x,w)) | x] \leq \Bb(x) && \forall x \in \Xx\backslash \Xx_u.
    \end{align}
\end{definition}

Observe that condition \eqref{eq:bar_cond_3} ensures that $\Bb$ acts as a supermartingale, \textit{i.e.}, the expected value of the function is nonincreasing at every time step. Definition \ref{def:bc} can be used to obtain the lower bound on the probability that the dt-SS $\Sys$ satisfies safety.

\begin{theorem}[Barrier certificates imply safety \cite{prajna2007framework}]
    \label{thm:bc}
    Consider a dt-SS $\Sys$ as in Definition \ref{def:system}. Let there exists a function $\Bb: \Xx \rightarrow \mathbb{R}$ for $\Sys$ such that it is a barrier certificate as in Definition \ref{def:bc} for some $0\leq \gamma \leq 1$. The probability that the solution process $\mathbf{x}_{x_0}$ starting from an initial state $x_0 \in \Xx_0$ does not reach unsafe states $\Xx_u$ is bounded by
    \begin{align}
        \Pro\{\mathbf{x}_{x_0}(t) \notin \Xx_u,\ \forall t \in \N\} \geq 1-\gamma.
        \label{eq:lowbndbc}
    \end{align}
\end{theorem}

We remark that Definition \ref{def:bc} is only useful if the initial set $\Xx_0$ and the unsafe set $\Xx_u$ are disjoint.

The standard strategy for finding barrier certificates starts by specifying a template, where the certificate is expressed as a linear combination of predetermined basis functions. 
For instance, if the template is a polynomial of a fixed degree, the corresponding basis functions are chosen as monomials. 
If the template instead is a neural network \cite{abate2021fossil}, one typically fixes the number of layers and nodes. 
Subsequently, search procedures such as Satisfiability Modulo Theory (SMT) solvers \cite{de2011satisfiability} or Sum-of-Squares (SOS) programming \cite{parrilo2003semidefinite} are employed to determine coefficients that satisfy conditions \eqref{eq:bar_cond_0}-\eqref{eq:bar_cond_3}. 
When no barrier certificate is obtained, a frequently used remedy is to modify the template. 
For example, one might increase the polynomial degree or alter the neural network architecture, thereby changing the template. 
However, such modifications typically increase the computational burden of the search and can still yield inconclusive outcomes. 
As an illustration, the authors of \cite{berger2024cone} are able to compute certificates in about $2$ minutes, but require up to $6$ hours—after which a time-out occurs—to verify these certificates with the SMT solver z3 \cite{de2008z3}. 
In what follows, we present, through the next example, an alternative method to address this difficulty.

\begin{example}[label=ex1]
\label{ex:ex1}
Consider the one-dimensional discrete-time state-space model
\begin{align}
    \label{eq:1d}
    \Sys: x(t+1) = 0.5x(t) + 0.05w(t).
\end{align}
The state space, initial set, and unsafe set are defined as $\Xx = [0,3]$, $\Xx_0 = [2,2.3]$, and $\Xx_u = [1.6,1.9]$, respectively.
We assume $w(t) \sim \mathcal{N}(0,1)$ and treat $\gamma$ as a decision variable, maximizing $1 - \gamma$ as a lower bound on the safety probability.
We then solve an SOS optimization problem based on conditions \eqref{eq:bar_cond_0}-\eqref{eq:bar_cond_3}. 
Using a degree-three polynomial template, the solver returns a minimum safety probability of $0.00624$ (i.e., $\gamma = 0.99376$). 
Increasing the polynomial degree to six yields a safety probability of at least $0.108$ (i.e., $\gamma = 0.892$). 
Further raising the degree to ten results in a safety probability lower bound of $0.389$ (i.e., $\gamma = 0.611$).
\end{example}

In the next section, we introduce several notions of multiple barrier certificates that incorporate additional hyperparameters, yielding less conservative conditions for establishing probabilistic safety guarantees.  
These notions, termed interpolation-inspired barrier certificates and $k$-inductive interpolation-inspired barrier certificates, extend the standard conditions, enable the replacement of a single complex template with multiple simpler ones (for instance, using several lower-degree polynomials instead of one higher-degree polynomial), and allow us to derive tighter probabilistic safety bounds for Example~\ref{ex:ex1}.

\section{Interpolation-Inspired and $k$-Inductive Barrier Certificates}
\label{sec:interp-k}

Here, we introduce a notion of interpolation-inspired barrier certificates and discuss their relation to $k$-inductive barrier certificates.
First, we discuss how ideas from (logical) interpolation may be used to consider different conditions for barrier certificates.
For a comprehensive explanation of (logical) interpolation related to inductive invariants in the context of hardware verification as discussed in \cite{mcmillan2003interpolation, bradley2011sat}, the interested reader is referred to Appendix~\ref{appendix:iii}.

\subsection{Interpolation-Inspired Barrier Certificate (IBC)}
\label{ss:ibc}

Here, we introduce a notion of interpolation-inspired barrier certificates (IBC) and demonstrate their efficacy.
The intuition behind these certificates is described in Appendix~\ref{appendix:ibc}.

\begin{definition}[IBC]
\label{def:ibc}
Consider a dt-SS $\Sys$ as in Definition \ref{def:system}. A set of functions $\Bb_i: \Xx \rightarrow \R$, for all $0\leq i\leq \ell$, is an IBC for $\Sys$ if there exist constants $0\leq \gamma \leq 1$, $\ell \in \N$, $\alpha_i \in \R_{> 0}$, $\ 0\leq i < \ell$ such that:

\begin{align}
    \label{eq:ibc0}
    & \Bb_i(x) \geq 0 && \forall x \in \Xx,\ 0 \leq i\leq \ell,\\
    \label{eq:ibc1}
    &\Bb_0(x) \leq \gamma && \forall x \in \Xx_0,\\
    \label{eq:ibc2}
    &\Bb_i(x) \geq 1 && \forall x \in \Xx_u,\ 0 \leq i\leq \ell,\\
    \label{eq:ibc3}
    &\E[\Bb_{i+1}(f(x,w)) | x] \leq \alpha_i \Bb_i(x)
    && \forall x \in \Xx \backslash \Xx_u,\ 0 \leq i < \ell,\\
    \label{eq:ibc4}
    &\E[\Bb_{\ell}(f(x,w)) | x] \leq \Bb_{\ell}(x) && \forall x \in \Xx \backslash \Xx_u.
\end{align}
\end{definition}

Definition \ref{def:ibc} can be used to obtain the lower bound on the probability that a dt-SS $\Sys$ is safe.

\begin{theorem}[IBCs imply safety]
    \label{thm:ibc}
    Consider a dt-SS $\Sys$ as in Definition \ref{def:system}. Let there exist a set of functions $\Bb_i: \Xx \rightarrow \mathbb{R}$, $\ 0\leq i\leq \ell$, for $\Sys$ such that it is an IBC as in Definition \ref{def:ibc} for some $0\leq \gamma \leq 1$, $\alpha_i \in \R_{> 0}$, $\ 0\leq i < \ell$. The probability that the solution process $\mathbf{x}_{x_0}$ starting from an initial state $x_0 \in \Xx_0$ does not reach unsafe set $\Xx_u$ is lower bounded by
    \begin{align}
        \Pro\{\mathbf{x}_{x_0}(t) \notin \Xx_u,~\forall t \in \N\} \geq 1-\gamma \left(1 + \prod_{i=0}^{\ell-1}\alpha_i + \sum_{t=0}^{\ell-2}\prod_{i = 0}^{t}\alpha_i\right).
        \label{eq:ibclowbnd}
    \end{align}
\end{theorem}

Observe that smaller values of $\gamma$ and  $\alpha_i$ provide better bounds. In particular, if we have $\alpha_i = 1$ for all $0 \leq i < \ell$, then the probability of safety is upper bounded by $1 - \gamma (1 + \ell)$. 
\begin{proof}
    Following condition \eqref{eq:ibc2}, we have $\Xx_u \subseteq \{x\in \Xx: \Bb_i(x) \geq 1, \ 0\leq i \leq \ell\}$. Now, we can separate the probability of visiting an unsafe state into visiting the unsafe state in less than and after $\ell$ time steps as follows:
    \begin{align}
        \Pro\{\exists t \in \N:& \mathbf{x}_{x_0}(t) \in \Xx_u\} \nonumber \\
        &\leq \Pro\{\exists 0\leq t< \ell: \mathbf{x}_{x_0}(t) \in \Xx_u\} \nonumber\\ 
        \label{eq:ibcproof1}
        &\hspace{1em}+ \Pro\{\exists t\geq \ell: \mathbf{x}_{x_0}(t) \in \Xx_u\} \\
        &\leq \Pro\{\exists 0\leq t< \ell: \Bb_t(x(t)) \geq 1\} \nonumber\\
        \label{eq:ibcproof2}
        &\hspace{1em}+ \Pro\{\exists t\geq \ell:\Bb_{\ell}(x(t)) \geq 1\}.
    \end{align}
    Each of these terms can be upper bounded with the use of Boole's inequality and Markov's inequality as follows:
    \begin{align}
        \label{eq:ibcproof3}
        \Pro\{& \exists 0\leq t < \ell:\Bb_t(x(t)) \geq 1\} \\
        &\leq \Pro \left\{\bigcup_{t=0}^{\ell-1} (\Bb_t(x(t)) \geq 1) \right\} \leq \sum_{t=0}^{\ell-1} \Pro\{\Bb_t(x(t)) \geq 1\} \\
        &\leq \E[\Bb_0(x(0))] + \sum_{t=1}^{\ell-1} \E[\Bb_t(x(t))]. 
    \end{align}

    Using the law of total expectation and condition \eqref{eq:ibc3} inductively, the expectations can be upper bounded as follows for $1\leq j \leq \ell$:
    \begin{align}
        \label{eq:exp1}
        \E[\Bb_j(x(j))] &= \E(\E[\Bb_j(x(j))|x(j-1)])\\
        \label{eq:exp2}
        &\leq \alpha_{j-1}\E[\Bb_{j-1}(x(j-1))] \leq \dots \\
        \label{eq:exp3}
        &\leq \E[\Bb_0(x(0))] \prod_{i=0}^{j-1} \alpha_i \leq \gamma \prod_{i=0}^{j-1} \alpha_i.
    \end{align}
    Then, it follows that:
    \begin{align}
        \label{eq:ibcupbnd1}
        \Pro\{\exists 0\leq t < \ell: \Bb_t(x(t)) \geq 1\} \leq \gamma + \gamma \sum_{t=0}^{\ell-2} \prod_{i=0}^{t} \alpha_i.
    \end{align}
    Conditions \eqref{eq:ibc0} and \eqref{eq:ibc4} show that $\Bb_{\ell}$ is a nonnegative supermartingale. By use of Ville's inequality and \eqref{eq:exp1}-\eqref{eq:exp3},
    \begin{align}
        \label{eq:ibcupbnd2}
        \Pro\{\exists t\geq \ell: \Bb_{\ell}(x(t)) \geq 1\} \leq \E[\Bb_{\ell}(x(\ell))] \leq \gamma \prod_{i=0}^{\ell-1}\alpha_i.
    \end{align}
    
    By complementation of the sum of \eqref{eq:ibcupbnd1} and \eqref{eq:ibcupbnd2} in \eqref{eq:ibcproof2}, we get the lower bound in \eqref{eq:ibclowbnd}.
\end{proof}

Note that by setting $\ell = 0$ in Definition \ref{def:ibc}, conditions \eqref{eq:ibc0}, \eqref{eq:ibc1}, \eqref{eq:ibc2} and \eqref{eq:ibc4} reduce to the standard barrier certificate conditions as in Definition \ref{def:bc} while condition \eqref{eq:ibc3} is no longer applicable. 
This is relevant for the implementation as we first start with $\ell = 0$ to find a standard barrier certificate.
We then increment $\ell$ by one only if we fail, and check for satisfaction of conditions \eqref{eq:ibc0}-\eqref{eq:ibc4}.
We repeat the above until we find an IBC or we reach a maximum number $\ell_{max}$. 
Any IBC found for $\ell > 0$ indicates that a standard barrier certificate with the given template could not be found. 
Also observe that once an IBC is found for a given $\ell \in \mathbb{N}$, we guarantee that an IBC can be found for all $j > \ell$. 
Thus, $\ell$ is the minimum integer that forms an IBC for a given fixed template.
However, it is possible that a larger $\ell$ can improve the safety probability.

We now demonstrate that, drawing on Example \ref{ex:ex1}, it is possible to construct an IBC that offers improved guarantees on the safety probability.
\begin{example}[continues=ex1]
We choose a family of cubic polynomials as the template for $\Bb_i(x)$. 
We then search for their coefficients, restricting the index to $i \leq \ell_{\max} = 3$, so that the resulting collection $\{\Bb_i(x)\}$ forms an IBC in the sense of Definition~\ref{def:ibc}. 
To determine these coefficients, we employ TSSOS \cite{wang2021tssos} in Julia to enforce conditions~\eqref{eq:ibc0}–\eqref{eq:ibc4} and solve for a feasible solution. 
The corresponding SOS program is detailed in Section~\ref{ss:ibcsynthesis}. 
In particular, $\gamma$ is treated as a decision variable, and the objective is to maximize the lower bound on the safety probability implied by Theorem~\ref{thm:ibc}.

With $\ell = 1$ and the choice $\alpha_0 = 0.7$, we obtain an IBC given by $\Bb_0(x) = -10.455x^3 + 77.183x^2 - 186.666x + 148.745$ and $\Bb_1(x) = 0.633x^3 - 0.849x^2 + 0.364x$. 
This solution yields $\gamma = 0.506$, which corresponds to a safety probability of at least $0.141$, already outperforming a single degree-six classical barrier certificate. 
Figure~\ref{fig:1d} illustrates the synthesized IBC along with the initial and unsafe sets. 
As required, $\Bb_0(x)$ remains in the interval $[0,\gamma]$ on the initial set and exceeds $1$ on the unsafe set. 
As required, $\Bb_1(x)$ is also greater than $1$ on the unsafe set and remains nonnegative throughout the entire state space.
Extending the IBC to three functions yields $\gamma = 0.333$, which increases the lower bound on the safety probability to $0.27$.

\begin{figure}[!t]
    \centering
    \epsfig{file=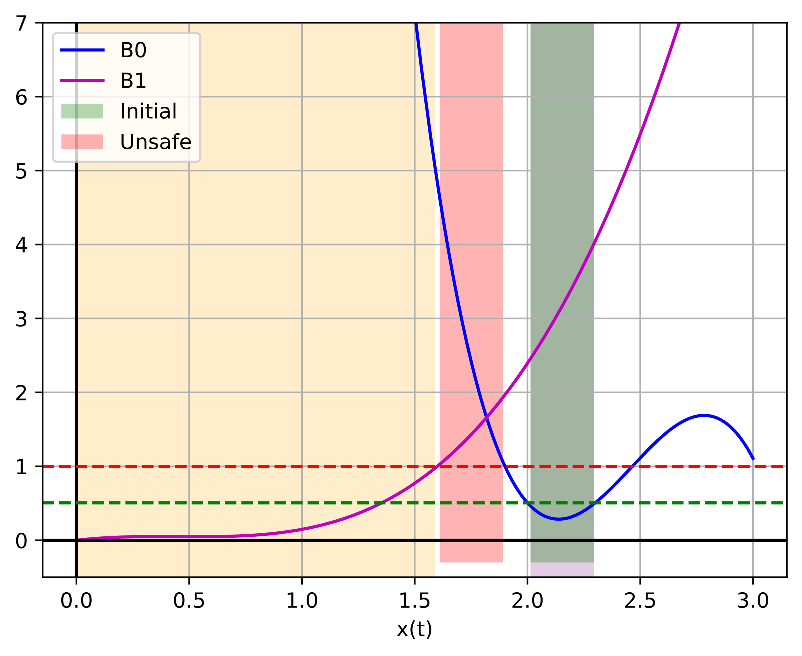, width=0.55\textwidth, keepaspectratio}
    \caption{IBC with $\ell = 1$. The purple (overlapping green) and orange shaded regions represent the sets $\{x: 0\leq \Bb_0(x)\leq \gamma\}$ and $\{x: 0\leq \Bb_1(x)<1\}$, respectively.
    The green and red dashed horizontal lines indicate $\gamma$ and $1$, respectively.}
    \label{fig:1d}
\end{figure}
\end{example}

Although our probabilistic guarantee for Example~\ref{ex:ex1} is weak, we demonstrate the advantages of our proposed conditions through more detailed case studies in Section~\ref{sec:case}.
The limited guarantee in this example arises from the fact that the reachable regions lie very close to the unsafe set, as illustrated in Figure~\ref{fig:1d}.

Our view of interpolation-inspired barrier certificates tackles a similar problem to that of $k$-inductive barrier certificates introduced in \cite{anand2022kstochastic}.
We discuss these similarities and differences in the following subsection.

\subsection{Relaxing $k$-Inductive Barrier Certificates ($k$-BCs)}
\label{ss:relaxk}

To discuss $k$-inductive barrier certificates ($k$-BCs), we first discuss some details on notation.
For a dt-SS $\Sys$ as in Definition \ref{def:system}, the value of the solution process after the $i^{th}$ time step is given by $x(t+i) = f^i(x(t),w_i(t))$, where $w_i(t) = [w(t);\dots;w(t+i-1)]$ is the vector containing all the noise terms from time $t$ to $t+i-1$.

As stated in \cite{anand2022kstochastic}, a function $\Bb: \Xx \rightarrow \R$ is a $k$-BC for dt-SS $\Sys$ for some constants $k\in \N_{\geq 1}, 0\leq \lambda_0 \leq 1$, and $c\geq 0$ if:
\begin{align}
    \label{eq:kold0}
    &\Bb(x)\geq 0 && \forall x \in \Xx,\\
    \label{eq:kold1}
    &\Bb(x)\leq \lambda_0 && \forall x\in \Xx_0,\\
    \label{eq:kold2}
    &\Bb(x)\geq 1 && \forall x\in \Xx_u,\\
    \label{eq:kold3}
    &\E[\Bb(f(x,w))|x] \leq \Bb(x)+c && \forall x\in \Xx \backslash \Xx_u,\\
    \label{eq:kold4}
    &\E[\Bb(f^k(x,w_k))|x] \leq \Bb(x) && \forall x\in \Xx \backslash \Xx_u.
\end{align}

We introduce the following definition of $k$-BCs, which relaxes the conditions compared to the one given above. In particular, the $c$-martingale condition is no longer required and is instead replaced by distinct functions. This formulation parallels that of IBCs, as both rely on multiple functions, thereby enabling simpler templates to serve as certificates. Nevertheless, the two formulations are, in general, not directly comparable. The intuition underlying this certificate construction is presented in Appendix~\ref{appendix:kbc}.

\begin{definition}
    \label{def:kbc}
    Consider a dt-SS $\Sys$ as in Definition \ref{def:system}. A set of functions $\Bb_i: \Xx \rightarrow \R$, $0\leq i <k$, is a $k$-BC for $\Sys$ if there exist constants $k\in \N_{\geq 1}$ and $\lambda_i \in \R_{>0}, 0\leq i <k$, such that:
    \begin{align}
    \label{eq:kbc0}
    & \Bb_i(x) \geq 0 && \forall x \in \Xx,\ 0 \leq i<k,\\
    \label{eq:kbc1}
    &\Bb_0(x) \leq \lambda_0 && \forall x \in \Xx_0,\\
    \label{eq:kbc2}
    &\Bb_i(x) \geq 1 && \forall x \in \Xx_u,\ 0 \leq i<k,\\
    \label{eq:kbc3}
    &\E[\Bb_{i}(f^i(x_0,w_i)) | x_0] \leq \lambda_i
    && \forall x_0 \in \Xx_0,\ 1 \leq i < k,\\
    \label{eq:kbc4}
    &\E[\Bb_i(f^k(x,w_k)) | x] \leq \Bb_{i}(x) && \forall x \in \Xx \backslash \Xx_u,0 \leq i < k.
    \end{align}
\end{definition}

The lower bound on the probability that the dt-SS $\Sys$ is safe can be derived from Definition \ref{def:kbc} using the following result.

\begin{theorem}
    \label{thm:kbc}
    Consider a dt-SS $\Sys$ as in Definition \ref{def:system}. Let there exist a set of functions $\Bb_i: \Xx \rightarrow \mathbb{R}$, $\ 0\leq i<k$, for $\Sys$ such that it is a $k$-BC as in Definition \ref{def:kbc} for some $k\in \N_{\geq 1}$ and $\lambda_i \geq 0, 0\leq i <k$. The probability that the solution process $\mathbf{x}_{x_0}$ starting from an initial state $x_0 \in \Xx_0$ does not reach unsafe set $\Xx_u$ is bounded by
    \begin{align}
        \label{eq:kbclowbnd}
        \Pro\{\mathbf{x}_{x_0}(t) \notin \Xx_u,\ \forall t \in \N\} \geq 1-\sum_{i=0}^{k-1}\lambda_i.
    \end{align}  
\end{theorem}

To obtain a meaningful probability, the sum of the $\lambda_i$ (and thus each individual $\lambda_i$) must be less than $1$. 
We can treat the $\lambda_i$ as decision variables and attempt to reduce them as much as possible. 
Note that doing so also tends to improve the achievable lower bound on the safety probability.

\begin{proof}
    Consider $k$ systems sampled after every $k$ steps, each starting from initial conditions $x(0), \dots, x(k - 1)$, respectively. The dynamic's equations are given as follows:
    \begin{align*}
        x(t+k) &= f^k(x(t), w_{k}(t)), \\
        x(t+k+1) &= f^k(x(t+1), w_{k}(t+1)), \\
        &\vdots\\
        x(t+2k-1) &= f^k(x(t+k-1), w_{k}(t+k-1)).
    \end{align*}
    Condition \eqref{eq:kbc4} implies that the $\Bb_{i}$ satisfy the supermartingale condition for each of these systems. Via  Boole's inequality and Ville's inequality, we get
    \begin{align}
         \Pro\{ \exists t& = i+jk, j\in \N, 0\leq i<k:\Bb_{i}(x(t))\geq1\} 
         \nonumber\\
         &\leq \sum_{i=0}^{k-1} \Pro\{\exists t = jk, j\in \N: \Bb_{i}(x(t+i)) \geq 1\}\\
         &\leq \E[B_{0}(x(0))] + \sum_{i=1}^{k-1} \E[B_{i}(x(i))].
    \end{align}
    Following conditions \eqref{eq:kbc1} and \eqref{eq:kbc3}, and using law of total expectation for each term on the right hand side of the inequality above, we have
    \begin{align*}
         \Pro\{\exists t\in N: \mathbf{x}_{x_0}(t) \in \Xx_u\}
         &\leq \lambda_0 +\sum_{i=1}^{k-1} \E(\E[B_{i}(f^i(x_0,w_i))|x_0])\\
         &\leq \lambda_0 +\sum_{i=1}^{k-1} \E(\lambda_i) = \sum_{i=0}^{k-1} \lambda_i.
    \end{align*}
    By complementation, we get the lower bound in \eqref{eq:kbclowbnd}.
\end{proof}

Note that a $k$-BC satisfying conditions \eqref{eq:kold0}-\eqref{eq:kold4} also satisfies conditions \eqref{eq:kbc0}-\eqref{eq:kbc4} by choosing $\Bb_i = \Bb$ and $\lambda_i = \lambda_0 + ic$. The lower bound on probability of safety via this choice is the same as the one given in \cite{anand2022kstochastic} based on the following simplification:
\begin{align*}
    \sum_{i=0}^{k-1} \lambda_i = \sum_{i=0}^{k-1} (\lambda_0 + ic) = k\lambda_0 + \frac{k(k-1)c}{2}.
\end{align*}
The relaxation of Definition \ref{def:kbc} follows from condition \eqref{eq:kbc3}, where the conditional is only over an initial state $x_0 \in \Xx_0$ while condition \eqref{eq:kold3} requires the inequality to hold for any given state $x \in \Xx\backslash \Xx_u$.

\subsection{Combining Interpolation and $k$-Induction}
\label{ss:combine}

We observe that although both $k$-BCs and IBCs offer comparable advantages, they are, in general, not directly comparable.
IBCs still impose a supermartingale requirement at every step, but they are defined using a collection of related functions that may not themselves be supermartingales.
In contrast, $k$-BCs require a supermartingale only at every $k^{\text{th}}$ step, while constraining the expected values by constants $\lambda_i$ for all $0 \leq i < k$.
Consequently, for a given template, it may happen that no $k$-BC is discovered, whereas an IBC is successfully obtained.
Indeed, for the system in Example \ref{ex:ex1}, we were unable to synthesize cubic $k$-BCs using SOS for $k \leq 4$.
This indicates that IBCs may be more suitable for certain systems.
Nevertheless, because both types of barrier certificates have been shown to admit simpler templates and to improve lower bounds on safety probabilities, one can systematically combine them into a unified approach.

As mentioned earlier, the function $\Bb_{\ell}$ from Definition \ref{def:ibc} is a nonnegative supermartingale for every time step starting at $\ell$. This condition could be restrictive and make finding suitable barrier certificates challenging. As discussed in \cite{anand2022kstochastic}, the supermartingale requirement at each time step for probabilistic safety guarantees could be relaxed for bounded-time horizon using $c$-martingale and combined with $k$-induction for unbounded time guarantees. 
Motivated by this relaxation, we combine the IBC formulation from Definition \ref{def:ibc} with the principle of $k$-induction to formulate a notion of what we call $k$-inductive interpolation-inspired barrier certificate ($k$-IBC).

We should note that there are many possible ways of  formulating $k$-IBCs. 
Let $\ell$ denote the number of functions considered for interpolation and $k$ be the bound on $k$-induction.
Then, the number of ways of combining them reduces to the number of ways of uniquely finding a supermartingale argument.
For $\ell$ number of functions in an IBC, we can apply $k$-induction to $m$ of these functions in ${\ell \choose m}$ ways, where $1\leq m \leq \ell$. Then for each of these $m$ options, we can select the last function to use for interpolation. This gives us a total of $\mathcal{O}\left(\sum_{m=1}^{\ell} {\ell \choose m} m\right)$ ways of combining interpolation and $k$-induction without considering potentially redundant formulations. 
However, not all of them may provide the same benefit, and further analysis is required to determine if a certain combination will lead to better probabilistic guarantees or simpler templates than others.
We now propose two notions of $k$-IBC.
 
The first notion of $k$-IBC derived from Definition \ref{def:ibc} and conditions \eqref{eq:kold3} and \eqref{eq:kold4} is defined as follows.
\begin{definition}[$k$-IBC v1]
    \label{def:kibc1}
    Consider a dt-SS $\Sys$ as in Definition \ref{def:system}. A set of functions $\Bb_i: \Xx \rightarrow \R, 0\leq i\leq \ell$, is a $k$-IBC for $\Sys$ if there exist constants $0\leq \gamma \leq 1$, $\ell \in \N$, $k \in \N_{\geq 1}$, $c \in \R_{\geq 0}$, $\alpha_i \in \R_{> 0}$, $0\leq i < \ell$, such that:
    \begin{align}
        \label{eq:kibc1_0}
        & \Bb_i(x) \geq 0 && \forall x \in \Xx,\ 0 \leq i\leq \ell,\\
        \label{eq:kibc1_1}
        &\Bb_0(x) \leq \gamma && \forall x \in \Xx_0,\\
        \label{eq:kibc1_2}
        &\Bb_i(x) \geq 1 && \forall x \in \Xx_u,\ 0 \leq i\leq \ell,\\
        \label{eq:kibc1_3}
        &\E[\Bb_{i+1}(f(x,w)) | x] \leq \alpha_i \Bb_i(x)
        && \forall x \in \Xx \backslash \Xx_u,\ 0 \leq i < \ell,\\
        \label{eq:kibc1_4}
        &\E[\Bb_{\ell}(f(x,w)) | x] \leq \Bb_{\ell}(x) + c && \forall x \in \Xx \backslash \Xx_u, \\
        \label{eq:kibc1_5}
        &\E[\Bb_{\ell}(f^k(x,w_k)) | x] \leq \Bb_{\ell}(x) && \forall x \in \Xx \backslash \Xx_u.
    \end{align}
\end{definition}

Note that condition \eqref{eq:kibc1_4} requires $\Bb_{\ell}$ to be a $c$-martingale at each time step and condition \eqref{eq:kibc1_5} requires $\Bb_{\ell}$ sampled after every $k^{th}$ step to be a supermartingale. Note that $c = 0$ gives us back IBC.
We now present the probabilistic bound of safety for $k$-IBC v1.
\begin{theorem}[$k$-IBC v1 implies safety]
    \label{thm:kibc1} 
    Consider a dt-SS $\Sys$ as in Definition \ref{def:system}. Let there exist a set of functions $\Bb_i: \Xx \rightarrow \mathbb{R}$, $\ 0\leq i\leq \ell$, for $\Sys$ such that it is a $k$-IBC as in Definition \ref{def:kibc1} for some $0\leq \gamma \leq 1$, $\ell \in \N$, $k \in \N_{\geq 1}$, $c \in \R_{\geq 0}$, $\alpha_i \in \R_{>0}$, $\ 0\leq i < \ell$. The probability that the solution process $\mathbf{x}_{x_0}$ starting from an initial state $x_0 \in \Xx_0$ does not reach unsafe set $\Xx_u$ is lower bounded by 
    \begin{align}
        \label{eq:kibc1lowbnd}
        \Pro\{\mathbf{x}_{x_0}(t) \notin \Xx_u,\ \forall t \in \N\} \geq 1- & \gamma \left(1 + k\prod_{i=0}^{\ell-1}\alpha_i + \sum_{t=0}^{\ell-2}\prod_{i = 0}^{t}\alpha_i \right) \nonumber\\
        &- \frac{k(k-1)c}{2}.
    \end{align}   
\end{theorem}

\begin{proof}
    \label{proof:kibc1}
    Expressions \eqref{eq:ibcproof1} and \eqref{eq:ibcproof2} still hold and the proof for less than $\ell$ time steps holds per \eqref{eq:ibcupbnd1}. For more than or equal to $\ell$ time steps, consider $k$ systems sampled after every $k$ steps, each starting from initial conditions $x(\ell), \dots, x(\ell + k - 1)$, respectively. The dynamic's equations are given as follows:
    \begin{align}
        \label{eq:dyn1}
        x(t+\ell+k) &= f^k(x(t+\ell), w_{k}(t+\ell)), \\
        x(t+\ell+k+1) &= f^k(x(t+\ell+1), w_{k}(t+\ell+1)), \\
        &\vdots \nonumber\\
        \label{eq:dynk}
        x(t+\ell+2k-1) &= f^k(x(t+\ell+k-1), w_{k}(t+\ell+k-1)).
    \end{align}
    Condition \eqref{eq:kibc1_5} implies that $\Bb_{\ell}$ satisfies the supermartingale condition for each of these systems. Via  Boole's inequality and Ville's inequality, we get
    \begin{align}
         \Pro\{&\exists t\geq \ell:\ \Bb_{\ell}(x(t)) \geq 1\} \nonumber\\ &\leq \sum_{i=0}^{k-1} \Pro\{\exists t = j\ell, j\in \N: \Bb_{\ell}(x(t+\ell+i)) \geq 1\}\\
         &\leq \sum_{i=0}^{k-1} \E[B_{\ell}(x(\ell+i))].
         \label{eq:kpbnd1}
    \end{align}
    Using the law of total expectation, condition \eqref{eq:kibc1_4} and expressions \eqref{eq:exp1}-\eqref{eq:exp3}, the expectations can be bounded as follows:\\ \vspace{-1.5em}
    \begin{align}
        \label{eq:kexp1}
        \E[B_{\ell}(x(\ell+i))] &= \E(\E[\Bb_{\ell}(x(\ell+i))|x(\ell+i-1)])\\
        \label{eq:kexp2}
        &\leq \E[\Bb_{\ell}(x(\ell+i-1))]+c \leq \dots \\
        \label{eq:kexp3}
        &\leq \E[\Bb_{\ell}(x(\ell))]  + ic \leq \gamma \prod_{j=0}^{\ell-1} \alpha_j + ic.
    \end{align}
    Then, it follows that:
    \begin{align}
        \Pro\{\exists t\geq \ell:\Bb_{\ell}(x(t)) \geq 1\} &\leq \sum_{i=0}^{k-1} \left(\gamma \prod_{j=0}^{\ell-1} \alpha_j + ic \right)\\
        &\leq k\gamma \prod_{i=0}^{\ell-1} \alpha_i + \frac{k(k-1)c}{2}.
        \label{eq:kpbnd2}
    \end{align}
    By complementation of the sum of \eqref{eq:ibcupbnd1} and \eqref{eq:kpbnd2} in \eqref{eq:ibcproof2}, we get the lower bound in \eqref{eq:kibc1lowbnd}.
\end{proof}

We next show, using Example \ref{ex:ex1}, that $k$-IBCs can further tighten the probabilistic lower bound on safety.
\begin{example}[continues=ex1]
\begin{figure}[!t]
    \centering
    \epsfig{file=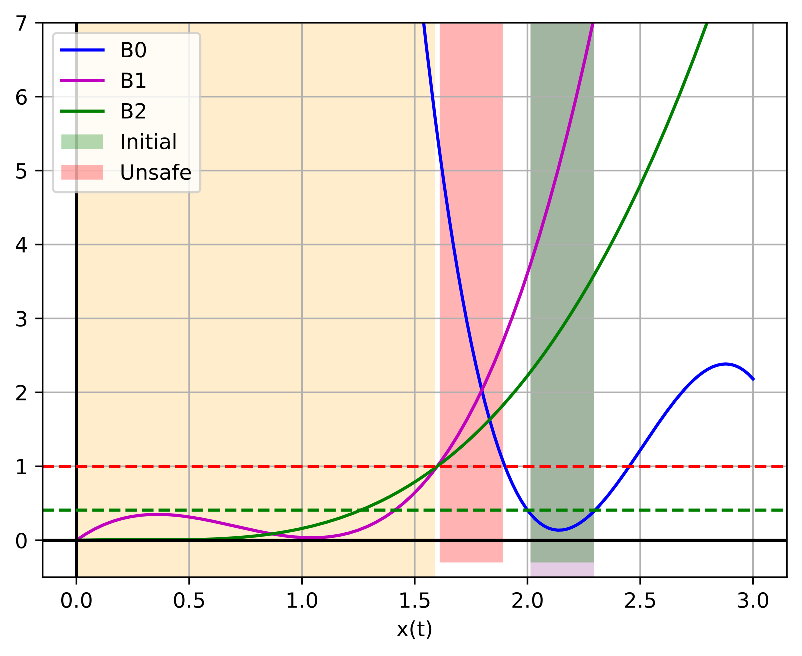, width=0.55\textwidth, keepaspectratio} 
    \caption{$k$-IBC v1 with $\ell = 2$, $k = 2$. The purple (overlapping green) and orange shaded regions represent the sets $\{x: 0\leq \Bb_0(x)\leq \gamma\}$ and $\bigcup_{i=1}^{2}\{x: 0\leq \Bb_i(x) < 1\}$, respectively. 
    The green and red dashed horizontal lines indicate $\gamma$ and $1$, respectively.}
    \label{fig:1d_kibc} 
\end{figure}
We take a family of cubic functions as the template for $\Bb_i(x)$. 
We then search for the corresponding coefficients, imposing $\ell_{max} = 3$ and $k_{max} = 3$, so that the resulting set of functions $\Bb_i(x)$ forms a $k$-IBC v1 in the sense of Definition~\ref{def:kibc1}. 
To determine these coefficients, we solve constraints~\eqref{eq:kibc1_0}-\eqref{eq:kibc1_5} using TSSOS~\cite{wang2021tssos} in Julia; the SOS formulation used is given in Section~\ref{ss:kibcsynthesis}. 
In particular, $\gamma$ and $c$ are treated as optimization variables, and the objective is to maximize the lower bound on the safety probability.

For $\ell = 2$, $k = 2$, and a uniform choice $\alpha_i = 0.3$ for all $0 \leq i < \ell$, we obtain a $k$-IBC v1 of the form
$\Bb_0(x) = -11.149x^3 + 83.922x^2 - 206.006x + 165.923$, 
$\Bb_1(x) = 1.969x^3 - 4.141x^2 + 2.209x$, and 
$\Bb_2(x) = 0.442x^3 - 0.375x^2 + 0.092x$. 
The resulting parameters are $\gamma = 0.408$ and $c = 0.00156$, which implies a safety probability of at least $0.395$.
This already improves on the lower bound obtained from a standard degree-ten barrier certificate.
Figure~\ref{fig:1d_kibc} illustrates the synthesized $k$-IBC v1 together with the initial and unsafe sets.
On the initial states, $\Bb_0(x)$ remains between $0$ and $\gamma$, while on the unsafe states, $\Bb_0(x)$ exceeds $1$. 
As required, both $\Bb_1(x)$ and $\Bb_2(x)$ are greater than $1$ on the unsafe set and remain nonnegative throughout the entire state space.
For $\ell = 3$, $k = 3$, and the same choice $\alpha_i = 0.3$, we obtain $\gamma = 0.342$ and $c = 0.00051$, which increases the lower bound on the safety probability to $0.495$.
\end{example}

Although the probabilistic safety guarantee for Example~\ref{ex:ex1} remains relatively modest, it is notably higher than the bound achieved using only an IBC, due to the additional strength provided by $k$-induction.

The second instance of $k$-IBC derived from Definitions \ref{def:ibc} and \ref{def:kbc} is defined as follows.
\begin{definition}[$k$-IBC v2]
    \label{def:kibc2}
    Consider a dt-SS $\Sys$ as in Definition \ref{def:system}. A set of functions $\Bb_i: \Xx \rightarrow \R$, for all $0\leq i< \ell+k$, is a $k$-IBC for $\Sys$ if there exist constants $0\leq \gamma \leq 1$, $\ell \in \N$, $k \in \N_{\geq 1}$, $\alpha_i \in \R_{> 0}$, $0\leq i < \ell$ and $\beta_j \in \R_{> 0}$, $1\leq j < k$, such that:
    \begin{align}
        \label{eq:kibc2_0}
        & \Bb_i(x) \geq 0 
        \hspace{5.8em} \forall x \in \Xx,\ 0 \leq i< \ell+k,\\
        \label{eq:kibc2_1}
        &\Bb_0(x) \leq \gamma 
        \hspace{5.7em} \forall x \in \Xx_0,\\
        \label{eq:kibc2_2}
        &\Bb_i(x) \geq 1 
        \hspace{5.8em} \forall x \in \Xx_u,\ 0 \leq i< \ell+k,\\
        \label{eq:kibc2_3}
        &\E[\Bb_{i+1}(f(x,w)) | x] \leq \alpha_i \Bb_i(x)\nonumber\\
        & \hspace{10em} \forall x \in \Xx \backslash \Xx_u,\ 0 \leq i < \ell,\\
        \label{eq:kibc2_4a}
        &\E[\Bb_{\ell+j}(f^{2j+1}(x,w_{2j+1})) | x] \leq \beta_j\Bb_{\ell-j-1}(x) \nonumber\\
        & \hspace{10em} \forall x \in \Xx \backslash \Xx_u, 1\leq j<\ell,\\
        \label{eq:kibc2_4b}
        &\E[\Bb_{\ell+j}(f^{\ell+j}(x,w_{\ell+j})) | x] \leq \beta_j\Bb_{0}(x) \nonumber\\
        & \hspace{10em} \forall x \in \Xx \backslash \Xx_u, \ell \leq j<k,\\
        \label{eq:kibc2_5}
        &\E[\Bb_{\ell+j}(f^{k}(x,w_{k})) | x] \leq \Bb_{\ell+j}(x)\nonumber\\
        & \hspace{10em} \forall x \in \Xx \backslash \Xx_u, 0\leq j<k.
    \end{align}
\end{definition}

Observe that in the above formulation, from Definition \ref{def:kbc}, 
$\lambda_0 = \alpha_{\ell-1}\Bb_{\ell-1}(x)$,
$\lambda_j = \beta_{j}\Bb_{\ell-j-1}(x)$ for $1\leq j< \ell$ or $\lambda_j = \beta_{j}\Bb_{0}(x)$ for $\ell \leq j< k$. Additionally, condition~\eqref{eq:kibc2_4b} is not applicable if $k < \ell$ and condition~\eqref{eq:kibc2_4a} will only apply for $1\leq j < k$. When $\ell$ steps are reached through interpolation, $k$-induction is utilized. The functions derived from interpolation are then employed to bound the functions involved in the $k$-induction process. Specifically, for this scenario, the expected value of $\Bb_{\ell + j}$ at the $(\ell + j)^{th}$ step is bounded by the value of $\Bb_{\ell-j-1}$ as stated in condition~\eqref{eq:kibc2_4a}. This results in a step difference of $\ell + j-(\ell-j-1) = 2j+1$ to transition from the state overestimated by $\Bb_{\ell-j-1}$ to that of $\Bb_{\ell + j}$, justifying the use of $2j+1$ in condition~\eqref{eq:kibc2_4a}. An analogous rationale explains the adoption of $\ell+j$ in condition~\eqref{eq:kibc2_4b}.

We now present the usefulness of $k$-IBC v2.
\begin{theorem}[$k$-IBC v2 implies safety]
    \label{thm:kibc2}
    Consider a dt-SS $\Sys$ as in Definition \ref{def:system}. Let there exist a set of functions $\Bb_i: \Xx \rightarrow \mathbb{R}$, $\ 0\leq i< \ell+k$, for $\Sys$ such that it is a $k$-IBC as in Definition \ref{def:kibc2} for some $0\leq \gamma \leq 1$, $\ell \in \N$, $k \in \N_{\geq 1}$, $\alpha_i \in \R_{> 0}$, $0\leq i < \ell$, and $\beta_j \in \R_{> 0}$, $1\leq j < k$. The probability that the solution process $\mathbf{x}_{x_0}$ starting from an initial state $x_0 \in \Xx_0$ does not reach unsafe set $\Xx_u$ is bounded by
    \begin{align}
        \Pro\{\mathbf{x}_{x_0}(t) \notin \Xx_u, \forall t \in \N\} \geq
        & 1 - \gamma \Bigg(1
        + \prod_{i=0}^{\ell-1} \alpha_i
        + \sum_{t=0}^{\ell-2} \prod_{i=0}^{t} \alpha_i \nonumber \\
        &+ \sum_{j=1}^{\ell-1} \Big(\beta_j  \prod_{i=0}^{\ell-j-1} \alpha_i \Big) 
        + \sum_{j=\ell}^{k-1} \beta_j \Bigg).
        \label{eq:kibc2lowbnd}
    \end{align}
\end{theorem}
\begin{proof}
    Once again, equation \eqref{eq:ibcproof1} still holds and the proof for less than $\ell$ time steps holds per \eqref{eq:ibcupbnd1}. For more than or equal to $\ell$ time steps, consider $k$ systems given in equations \eqref{eq:dyn1}-\eqref{eq:dynk}.
    Condition \eqref{eq:kibc2_5} implies that each $\Bb_{\ell+j}$ satisfies the supermartingale condition for each of these systems. Via  Boole's inequality and Ville's inequality, we get
    \begin{align}
         \Pro\{&\exists t = i\ell + j, i\in \N_{\geq 1}, 0\leq j<k: \Bb_{\ell+j}(x(t)) \geq 1\} \nonumber\\ &\leq \sum_{j=0}^{k-1} \Pro\{\exists t = i\ell, i\in \N: \Bb_{\ell+j}(x(t+\ell+j)) \geq 1\}\\
         \label{eq:kp2_bnd1}
         &\leq \E[B_{\ell}(x(\ell))] + \sum_{j=1}^{k-1} \E[B_{\ell+j}(x(\ell+j))].
    \end{align}
    The first term from the right hand side of the above inequality can be upper bounded using inequality \eqref{eq:exp3}. Each term of the summation can be upper bounded using conditions \eqref{eq:kibc2_4a}, \eqref{eq:kibc2_4b},  inequality \eqref{eq:exp3} and law of total expectation as follows.
    For $1\leq j<\ell$,
    \begin{align}
        \label{eq:kexp2_bnd1}
        \E[&B_{\ell+j}(x(\ell+j))] \nonumber\\
        &= \E(\E[\Bb_{\ell+j}(f^{2j+1}(x(\ell-j-1), w_{2j+1})|x(\ell-j-1)]) \\
        &\leq \beta_j \E[B_{\ell-j-1}(x(\ell-j-1))] \leq \gamma \beta_j  \prod_{i=0}^{\ell-j-1} \alpha_i.
    \end{align}
    For $\ell\leq j<k$,
    \begin{align}
        \label{eq:kexp2_bnd2}
        \E[B_{\ell+j}(x(\ell+j))] &= \E(\E[\Bb_{\ell+j}(f^{\ell+j}(x(0), w_{\ell+j})|x(0)]) \\
        &\leq \beta_j \E[B_{0}(x(0))] \leq \gamma \beta_j.
    \end{align}
    Then, 
    \begin{align}
        \label{eq:kp2_bnd2}
        \sum_{j=0}^{k-1} \Pro\{&\exists t\geq \ell: \Bb_{\ell+j}(x(t)) \geq 1\} \nonumber\\ 
        &\leq \gamma \prod_{i=0}^{\ell-1} \alpha_i 
        + \gamma \sum_{j=1}^{\ell-1} \left(\beta_j  \prod_{i=0}^{\ell-j-1} \alpha_i \right)
        + \gamma \sum_{j=\ell}^{k-1} \beta_j.
    \end{align}
    By complementation of the sum of \eqref{eq:ibcupbnd1} and \eqref{eq:kp2_bnd2} in \eqref{eq:ibcproof1}, we get the lower bound in \eqref{eq:kibc2lowbnd}.
\end{proof}

\begin{remark}
Note that for a $k$-IBC, a choice of $\ell = 0$, $k = 1$ boils down to a standard barrier certificate, a choice of $k = 1$ boils down to an IBC and a choice of $\ell = 0$ boils down to a $k$-BC.
\end{remark}

\section{Synthesizing IBC and $k$-IBC}
\label{sec:synthesis}

Here, we provide computational methods for synthesizing IBCs and $k$-IBCs based on Definitions \ref{def:ibc} and \ref{def:kibc1}, respectively. 
To do so, we first note that a set $V \subseteq \R^n$ is semi-algebraic if it can be defined with a vector of polynomial inequalities of $h(x)$ as $V = \{ x\in \R^n : h(x) \geq 0 \}$, where the inequalities are element-wise.

The technique of using semidefinite programming \cite{parrilo2003semidefinite} and framing the search for standard barrier certificates as SOS polynomials \cite{prajna2007framework} is usually simpler, scales relatively better and takes less time when compared to SMT based approaches.
In this section, we provide an SOS formulation as we found it to be the most effective for our case studies.

To do so, we make the following assumption over dt-SS $\Sys$.

\begin{assumption}
    The dt-SS $\Sys$ has a continuous state set $\Xx \subseteq \R^n$, and its transition function $f : \Xx \times \Ww \rightarrow \Xx$ is a polynomial function of the state variable $x$ and noise variable $w$. The sets $\Xx, \Xx_0$ and $\Xx_u$ are semi-algebraic with corresponding vectors of polynomials $g(x), g_0(x)$ and $g_u(x)$, respectively.
    \label{asp:sos}
\end{assumption}

We now show how one may utilize a SOS approach to find IBCs and $k$-IBCs. Note that we put maximizing the safety probability as an objective.

\subsection{IBCs}
\label{ss:ibcsynthesis}

Under Assumption \ref{asp:sos}, the IBC conditions \eqref{eq:ibc0}-\eqref{eq:ibc4} can be formulated as a set of SOS constraints in order to compute a polynomial IBC of a predefined degree per the following lemma.
\begin{lemma}
    \label{lem:sosIBC}
    Consider a dt-SS $\Sys$. Suppose Assumption \ref{asp:sos} holds for $\Sys$. Suppose there exist constants $\ell \in \N$, $\alpha_i \in \R_{>0}$, $0\leq i < \ell$, polynomials of same degree $\Bb_i(x)$ and SOS polynomials $\ \hat{\eta}_i(x), \eta_0(x),\ \eta_{u,i}(x),\ \eta_i(x),\ \hat{\eta}(x)$ of appropriate dimensions. The objective-based SOS optimization problem is given as follows:
    \begin{align}
        \max_{0 \leq \gamma \leq 1}\ & p = 1-\gamma \left(1 + \prod_{i=0}^{\ell-1}\alpha_i + \sum_{t=0}^{\ell-2}\prod_{i = 0}^{t}\alpha_i\right) && \nonumber\\
        \mathsf{s.t.}\ & 0 \leq p \leq 1,&& \\
        &\text{The following are SOS polynomials:} && \nonumber\\
        &\Bb_i(x) - \hat{\eta}_i^T(x) g(x) 
        \hspace{6.5em} \forall~ 0\leq i \leq \ell,\\
        &\gamma - \Bb_0(x) - \eta_0^T(x) g_0(x),\\
        &\Bb_i(x) - 1 - \eta_{u,i}^T(x) g_u(x) 
        \hspace{4em} \forall~ 0\leq i \leq \ell,\\
        &\alpha_i \Bb_{i}(x) - \E[\Bb_{i+1}(f(x,w))|x] - \eta_i^T(x) g(x) \nonumber\\
        & \hspace{14.5em}\forall~ 0\leq i < \ell,\\
        &\Bb_{\ell}(x) - \E[\Bb_{\ell}(f(x,w))|x] - \hat{\eta}^T(x) g(x),
    \end{align}
    where $x$ is the state variable over the set $\Xx$ and $w$ is the noise variable over the set $\Ww$. Then, the set of functions $\Bb_i(x),\ 0\leq i\leq \ell$, is an IBC following Definition \ref{def:ibc}.
\end{lemma}

\subsection{$k$-IBCs}
\label{ss:kibcsynthesis}
The SOS formulations for the two instances of $k$-IBCs discussed in Section \ref{ss:combine} are given as follows.

\subsubsection{$k$-IBC v1}
Under Assumption \ref{asp:sos}, the $k$-IBC v1 conditions \eqref{eq:kibc1_0}-\eqref{eq:kibc1_5} can be formulated as a set of SOS constraints per the following lemma.

\begin{lemma}
    \label{lem:soskibc1}
    Consider a dt-SS $\Sys$. Suppose Assumption \ref{asp:sos} holds for $\Sys$. Suppose there exist constants $\ell \in \N$, $k \in \N_{\geq 1}$, $\alpha_i \in \R_{>0}$, $0\leq i < \ell$, polynomials of same degree $\Bb_i(x)$ and SOS polynomials $\hat{\eta}_i(x)$, $\eta_0(x)$, $\eta_{u,i}(x)$, $\eta_i(x)$, $\hat{\eta}(x)$, $\hat{\eta}_k(x)$ of appropriate dimensions.
    The objective-based SOS optimization problem is given as follows:
    \begin{align}
        \max_{\substack{0 \leq \gamma \leq 1 \\ c \geq 0}}\ & p = 1 - \gamma \left(1 + k\prod_{i=0}^{\ell-1}\alpha_i + \sum_{t=0}^{\ell-2}\prod_{i = 0}^{t}\alpha_i \right) - \frac{k(k-1)c}{2} && \nonumber\\
        \mathsf{s.t.}\ & 0 \leq p \leq 1,&& \\
        &\text{The following are SOS polynomials:} && \nonumber\\
        &\Bb_i(x) - \hat{\eta}_i^T(x) g(x) 
        \hspace{6.5em} \forall~0\leq i \leq \ell,\\
        &\gamma - \Bb_0(x) - \eta_0^T(x) g_0(x),\\
        &\Bb_i(x) - 1 - \eta_{u,i}^T(x) g_u(x) 
        \hspace{4em} \forall~0\leq i \leq \ell,\\
        &\alpha_i \Bb_{i}(x) - \E[\Bb_{i+1}(f(x,w))|x] - \eta_i^T(x) g(x)\nonumber\\
        & \hspace{14.5em} \forall~0\leq i < \ell,\\
        &\Bb_{\ell}(x) + c - \E[\Bb_{\ell}(f(x,w))|x] - \hat{\eta}^T(x) g(x),\\
        &\Bb_{\ell}(x) - \E[\Bb_{\ell}(f^k(x,w))|x] - \hat{\eta}_k^T(x) g(x),
    \end{align}
    where $x$ is the state variable over the set $\Xx$ and $w$ is the noise variable over the set $\Ww$. Then the set of functions $\Bb_i(x),\ 0\leq i\leq \ell$, is a $k$-IBC v1 following Definition \ref{def:kibc1}.
\end{lemma}

\subsubsection{$k$-IBC v2}
Similarly, under Assumption \ref{asp:sos}, the $k$-IBC v2 conditions \eqref{eq:kibc2_0}-\eqref{eq:kibc2_5} can be formulated as a set of SOS constraints per the following lemma.
\begin{lemma}
    \label{lem:soskibc2}
    Consider a dt-SS $\Sys$. Suppose Assumption \ref{asp:sos} holds for $\Sys$. Suppose there exist constants $\ell \in \N$, $k \in \N_{\geq 1}$, $\alpha_i \in \R_{> 0}$, $0\leq i < \ell$, and $\beta_j \in \R_{> 0}$, $1\leq j < k$, polynomials of same degree $\Bb_i(x)$ and SOS polynomials $\hat{\eta}_i(x)$, $\eta_0(x)$, $\eta_{u,i}(x)$, $\eta_i(x)$, $\eta_{\ell,j}(x)$, $\hat{\eta}_{j,k}(x)$ of appropriate dimensions. 
    The objective-based SOS optimization problem is given as follows:    
    \begin{align}
        \max_{0 \leq \gamma \leq 1}\ & p  = 1 - \gamma \Bigg(1
        + \prod_{i=0}^{\ell-1} \alpha_i
        + \sum_{t=0}^{\ell-2} \prod_{i=0}^{t} \alpha_i 
        + \sum_{j=1}^{\ell-1} \Big(\beta_j  \prod_{i=0}^{\ell-j-1} \alpha_i \Big) 
        + \sum_{j=\ell}^{k-1} \beta_j \Bigg) \nonumber\\
        \mathsf{s.t.}\ & 0 \leq p \leq 1, \\
        &\text{The following are SOS polynomials:} \nonumber\\
        &\Bb_i(x) - \hat{\eta}_i^T(x) g(x) 
        \hspace{5em} \forall~ 0\leq i < \ell+k,\\
        &\gamma - \Bb_0(x) - \eta_0^T(x) g_0(x),\\
        &\Bb_i(x) - 1 - \eta_{u,i}^T(x) g_u(x) 
        \hspace{2.5em} \forall~ 0\leq i < \ell+k,\\
        &\alpha_i \Bb_{i}(x) - \E[\Bb_{i+1}(f(x,w))|x] - \eta_i^T(x) g(x) \nonumber\\
        & \hspace{13em} \forall~0\leq i < \ell,\\
        &\beta_j\Bb_{\ell-j-1}(x) - \E[\Bb_{\ell+j}(f^{2j+1}(x,w_{2j+1})) | x] - \eta_{\ell,j}^T(x)g(x) \nonumber\\
        & \hspace{13em} \forall~ 1\leq j<\ell,\\
        &\beta_j\Bb_{0}(x) - \E[\Bb_{\ell+j}(f^{\ell+j}(x,w_{\ell+j})) | x] - \eta_{\ell,j}^T(x)g(x) \nonumber\\
        & \hspace{13em} \forall~\ell \leq j<k,\\
        &\Bb_{\ell+j}(x) - \E[\Bb_{\ell+j}(f^{k}(x,w_{k})) | x] - \hat{\eta}_{j,k}^T(x)g(x)\nonumber\\
        & \hspace{13em} \forall~ 0\leq j<k,
    \end{align}
    where $x$ is the state variable over the set $\Xx$ and $w$ is the noise variable over the set $\Ww$. Then the set of functions $\Bb_i(x),\ 0\leq i < \ell +k$, is a $k$-IBC v2 following Definition \ref{def:kibc2}.
\end{lemma}

Observe that our formulation alleviates the computational complexity issues typically encountered with SOS in the following manner:
for classical SOS-based barrier certificates, the search problem has polynomial complexity on the order of $O\big(\binom{n+d}{d} \times \binom{n+d}{d}\big)$ \cite[Theorem 3.3]{parrilo2003semidefinite}, where $n$ is the system dimension and $2d$ denotes the degree of the SOS polynomial.
The factor $O\big(\binom{n+d}{d} \times \binom{n+d}{d}\big)$ represents the number of decision variables introduced when the SOS problem is reformulated as an SDP, and this number grows polynomially with the degree $2d$.
By contrast, for interpolation-based barrier certificates, our framework enables the use of lower-degree functions as certificates, and in this setting the additional complexity is only constant, i.e., $O(1)$, as long as the number of such functions is kept fixed.
If we allow $\ell$ to vary, the complexity scales only linearly with $\ell$, that is, it becomes $O(\ell)$ instead of polynomial in $\ell$.
Thus, one can still exploit existing methods while substantially decreasing the computational burden associated with searching for a certificate.

\section{Case Studies}
\label{sec:case}

We illustrate the effectiveness of IBCs and $k$-IBCs on both a Lotka–Volterra-type model and a four-dimensional system model. For a state $x = x(t)$, we write $x'$ to represent $x(t+1)$. A comparison of the different certificates is reported in Table~\ref{tab:compare}. The resulting certificates can be accessed via this link.~\footnote{\href{https://drive.google.com/file/d/1if7aiR8YapQ2NusxLf-h5JtLcvSEJk_z/view?usp=sharing}{Case study certificates}.
}

\subsection{2D Lotka-Volterra Model}
\label{ss:lotka}

For our first case study, we consider the discrete-time Lotka-Volterra type prey–predator model with state variables $v, p$ denoting the victim/prey and the predator, respectively. The dynamics is given by the following difference equations:
\begin{align*}
\begin{cases}
    v'= v + T(\theta v(1-v) - \phi vp) + G_0w,\\
    p'= p - T(\psi p - \delta vp) + G_1w,
\end{cases}
\end{align*}
where $T = 0.1s$ is the sampling time, $\theta = 1.1$ is the growth rate of the prey, $\phi = 0.4$ is the death rate of the prey, $\psi = 0.4$ is the death rate of the predator, $\delta = 0.1$ is the growth rate of the predator, and $G_0 = 0.01$, $G_1 = 0.005$ are the noise coefficients. 
The state set, initial set, and unsafe set are given by $\Xx = [0,10]\times[0,5],\Xx_0 = [6,7]\times[2,3]$, and $\Xx_u = [3,5]\times[0,3]$, respectively. 
We first consider a degree-five polynomial function in two variables as our parametric template of the barrier certificate $\Bb(v,p)$ and attempt to compute suitable coefficients such that $\Bb(v,p)$ is a standard barrier certificate as in Definition~\ref{def:bc} (i.e. IBC with $\ell= 0$). 
We used TSSOS~\cite{wang2021tssos} within Julia to reformulate conditions~\eqref{eq:bar_cond_0}-\eqref{eq:bar_cond_3} as SOS optimization problem as described in the previous section. 
However, we were unable to find a standard barrier certificate.

We subsequently expressed conditions~\eqref{eq:ibc0}-\eqref{eq:ibc4} as an SOS optimization problem using Lemma~\ref{lem:sosIBC}. In this formulation, we chose $\alpha_i = 0.44$, set $\ell_{max} = 3$, and adopted the same parametric structure for $\Bb_i(v,p)$ as in the previous subsection. 
This procedure yielded an IBC characterized by $\ell = 1$ and $\gamma = 0.1$. The resulting functions are depicted in Figure~\ref{fig:ibc}. 
In the figure, the blue and purple shaded areas represent the sublevel sets of $\Bb_0(v,p)$ and $\Bb_1(v,p)$, respectively. 

\begin{figure}[!t]
    \centering
    \epsfig{file=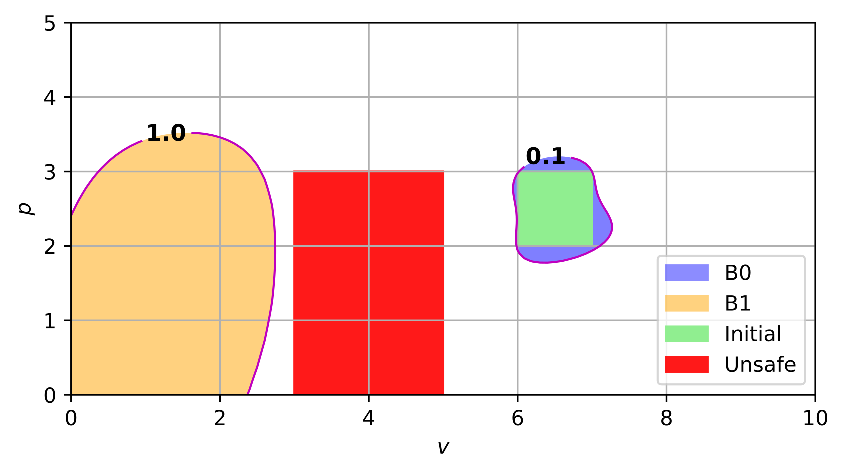, width=0.55\textwidth, keepaspectratio}
    \caption{IBC with $\ell = 1$ for Lotka-Volterra type model. The axes show the state variables $v$ and $p$. The blue and purple shaded regions show the sets $\{x: 0\leq \Bb_0(v,p)\leq \gamma\}$ and $\{x: 0\leq \Bb_1(v,p)<1\}$, respectively.} 
    \label{fig:ibc}
\end{figure}

\subsection{4D System}
\label{ss:4d}

For our second case study, we examine the following four-dimensional dynamical system, adapted from~\cite{bhattacharjee2024discrete}:
\begin{align*}
&\begin{cases}
    x_1' = x_1 + T(r x_2 - (b + d_1) x_1 -h x_1 x_4) + G_0 w\\
    x_2' = x_2 + T(b x_1 - b x_2^2 - h x_2 x_4 - d_2 x_2) + G_0 w\\
    x_3' = x_3 + T(\alpha h x_1 x_4 + \beta h x_2 x_4 - d_3 x_3 -\xi x_2 x_4) + G_1 w\\
    x_4' = x_4 + T (n x_3 - d_4 x_4) + G_1 w,
\end{cases}
\end{align*}
with parameters $r = 1.6$, $b = 0.3$, $h = 20$, $\alpha = 0.2$, $\beta = 0.2$, $\xi = 0.08$, $n = 0.3$, $d_1 = 0.08$, $d_2 = 0.06$, $d_3 = 0.8$, $d_4 = 0.5$, $T = 0.01$, $G_0 = 0.005$, and $G_1 = 0.001$. 
The state space, initial region, and unsafe region are given by $\Xx = [0,10]^4$, $\Xx_0 = [6.5,7]\times[5.5,6]\times[4.5,5]\times[3.5,4]$, and $\Xx_u = [3,5]\times[3,4]\times[0,8]\times[0,5]$, respectively. 
As a parametric template for the barrier certificate $\Bb(x_1,x_2,x_3,x_4)$, we use multivariate polynomials in four variables and search for coefficients that render $\Bb(x_1,x_2,x_3,x_4)$ a valid certificate. 
Observe that a $k$-IBC v1 with $\ell = 0$, $k = 1$ coincides with a classical barrier certificate; for $\ell = 0$ it specializes to a $k$-BC; and for $k = 1$ it reduces to an IBC. 
We utilize TSSOS~\cite{wang2021tssos} in Julia to cast conditions~\eqref{eq:kibc1_0}-\eqref{eq:kibc1_5} into an SOS optimization problem via Lemma~\ref{lem:soskibc1}. 
We choose $\ell_{max} = 2$, $k_{max} = 3$, $degree_{max} = 8$, and sweep $\alpha_i$ over the interval $[0.3,1]$. 
With these settings, we obtain a degree-five $k$-IBC v1 with $\ell = 1$, $k = 2$, $\gamma = c = 10^{-4}$, and all $\alpha_i = 0.4$, which yields a safety probability lower bound of $0.9997$.

\begin{table}[!t]
\centering
\caption{Summary of certificates for the case studies.}
\begin{tabular}{ |p{0.07\linewidth}|p{0.23\linewidth}|p{0.13\linewidth}|p{0.16\linewidth}|p{0.1\linewidth}|  }
 \hline
System & Method & Degree & Computation Times (s) & Probability\\
\hline
 \multirow{2}{*}{2D}  & BC & NF ($\leq 8$) & $-$ & $-$\\
 & $k$-BC ($k \leq 3$) & NF ($\leq 6$) & $-$ & $-$\\
 & IBC ($\ell = 1$) & $5$ & $13.80$ & $0.856$\\
 \hline
 \multirow{3}{*}{4D}  & BC & NF ($\leq 8$) & $-$ & $-$\\
 & $k$-BC ($k \leq 3$) & NF ($\leq 5$) & $-$ & $-$\\
 & IBC ($\ell = 2$) & $6$ & $207.23$ & $0.9997$\\
 & $k$-IBC~v1 ($\ell = 1$, $ k = 2$) & $5$ & $1739.42$ & $0.9997$\\
 \hline
\end{tabular}
{\raggedright \vspace{0.1em} * NF = Not Found \par}
\label{tab:compare}
\end{table}

\section{Conclusion}
\label{sec:conclusion}
We proposed a notion of interpolation-inspired barrier certificate (IBC) and $k$-inductive interpolation-inspired barrier certificate ($k$-IBC) for stochastic dynamical systems.
These certificates relax the conditions of a standard barrier certificate by incrementally finding functions that together guarantee safety. 
We presented SOS optimization as a prominent technique of computing this set of functions under mild assumptions.
Using an example and case studies, we demonstrated that given a barrier certificate template, one may find IBC and $k$-IBC with better lower bound on safety probability even when standard barrier certificates could not be computed. 
Given that SOS-based approaches do not computationally perform well for systems with high dimensions, we hope that the potential to find multiple low degree polynomials via IBC and $k$-IBC will alleviate these concerns. 
As future work, we plan to investigate data-driven and neural network based approaches to find IBCs and $k$-IBCs as well as explore their use in controller synthesis.
We also plan to investigate how to automate the search for combinations of interpolation and $k$-induction with the help of counterexamples.

\bibliographystyle{alpha}
\bibliography{reference.bib}

@article{jagtap_2020_formal,
  title={Formal synthesis of stochastic systems via control barrier certificates},
  author={Jagtap, Pushpak and Soudjani, Sadegh and Zamani, Majid},
  journal={IEEE Transactions on Automatic Control},
  volume={66},
  number={7},
  pages={3097--3110},
  year={2020},
  publisher={IEEE}
}

@inproceedings{prajna2004safety,
  title={Safety verification of hybrid systems using barrier certificates},
  author={Prajna, Stephen and Jadbabaie, Ali},
  booktitle={International Workshop on Hybrid Systems: Computation and Control},
  pages={477--492},
  year={2004},
  organization={Springer}
}

@inproceedings{prajna2004stochastic,
  title={Stochastic safety verification using barrier certificates},
  author={Prajna, Stephen and Jadbabaie, Ali and Pappas, George J},
  booktitle={2004 43rd IEEE conference on decision and control (CDC)(IEEE Cat. No. 04CH37601)},
  volume={1},
  pages={929--934},
  year={2004},
  organization={IEEE}
}

@article{prajna2007framework,
  title={A framework for worst-case and stochastic safety verification using barrier certificates},
  author={Prajna, Stephen and Jadbabaie, Ali and Pappas, George J},
  journal={IEEE Transactions on Automatic Control},
  volume={52},
  number={8},
  pages={1415--1428},
  year={2007},
  publisher={IEEE}
}

@inproceedings{mcmillan2003interpolation,
  title={Interpolation and SAT-based model checking},
  author={McMillan, Kenneth L},
  booktitle={Computer Aided Verification: 15th International Conference, CAV 2003, Boulder, CO, USA, July 8-12, 2003. Proceedings 15},
  pages={1--13},
  year={2003},
  organization={Springer}
}

@article{de2011satisfiability,
  title={Satisfiability modulo theories: introduction and applications},
  author={De Moura, Leonardo and Bj{\o}rner, Nikolaj},
  journal={Communications of the ACM},
  volume={54},
  number={9},
  pages={69--77},
  year={2011},
  publisher={ACM New York, NY, USA}
}

@inproceedings{bradley2011sat,
  title={SAT-based model checking without unrolling},
  author={Bradley, Aaron R},
  booktitle={International Workshop on Verification, Model Checking, and Abstract Interpretation},
  pages={70--87},
  year={2011},
  organization={Springer}
}

@inproceedings{bradley2012understanding,
  title={Understanding ic3},
  author={Bradley, Aaron R},
  booktitle={International Conference on Theory and Applications of Satisfiability Testing},
  pages={1--14},
  year={2012},
  organization={Springer}
}

@book{troelstra2000basic,
  title={Basic proof theory},
  author={Troelstra, Anne Sjerp and Schwichtenberg, Helmut},
  number={43},
  year={2000},
  publisher={Cambridge University Press}
}

@inproceedings{de2008z3,
  title={Z3: An efficient SMT solver},
  author={De Moura, Leonardo and Bj{\o}rner, Nikolaj},
  booktitle={International conference on Tools and Algorithms for the Construction and Analysis of Systems},
  pages={337--340},
  year={2008},
  organization={Springer}
}

@article{wang2021tssos,
  title={TSSOS: A moment-SOS hierarchy that exploits term sparsity},
  author={Wang, Jie and Magron, Victor and Lasserre, Jean-Bernard},
  journal={SIAM Journal on optimization},
  volume={31},
  number={1},
  pages={30--58},
  year={2021},
  publisher={SIAM}
}

@inproceedings{anand2021safety,
  title={Safety verification of dynamical systems via k-inductive barrier certificates},
  author={Anand, Mahathi and Murali, Vishnu and Trivedi, Ashutosh and Zamani, Majid},
  booktitle={2021 60th IEEE Conference on Decision and Control (CDC)},
  pages={1314--1320},
  year={2021},
  organization={IEEE}
}

@article{cabodi2008strengthening,
  title={Strengthening model checking techniques with inductive invariants},
  author={Cabodi, Gianpiero and Nocco, Sergio and Quer, Stefano},
  journal={IEEE transactions on computer-aided design of integrated circuits and systems},
  volume={28},
  number={1},
  pages={154--158},
  year={2008},
  publisher={IEEE}
}

@inproceedings{zhang2004incremental,
  title={Incremental deductive \& inductive reasoning for SAT-based bounded model checking},
  author={Zhang, Liang and Prasad, Mukul R and Hsiao, Michael S},
  booktitle={IEEE/ACM International Conference on Computer Aided Design, 2004. ICCAD-2004.},
  pages={502--509},
  year={2004},
  organization={IEEE}
}

@inproceedings{anand2022kstochastic,
  title={K-inductive barrier certificates for stochastic systems},
  author={Anand, Mahathi and Murali, Vishnu and Trivedi, Ashutosh and Zamani, Majid},
  booktitle={Proceedings of the 25th ACM International Conference on Hybrid Systems: Computation and Control},
  pages={1--11},
  year={2022}
}

@article{parrilo2003semidefinite,
  title={Semidefinite programming relaxations for semialgebraic problems},
  author={Parrilo, Pablo A},
  journal={Mathematical programming},
  volume={96},
  pages={293--320},
  year={2003},
  publisher={Springer}
}

@inproceedings{sogokon2018vector,
  title={Vector barrier certificates and comparison systems},
  author={Sogokon, Andrew and Ghorbal, Khalil and Tan, Yong Kiam and Platzer, Andr{\'e}},
  booktitle={International Symposium on Formal Methods},
  pages={418--437},
  year={2018},
  organization={Springer}
}

@inproceedings{anand2019verification,
  title={Verification of switched stochastic systems via barrier certificates},
  author={Anand, Mahathi and Jagtap, Pushpak and Zamani, Majid},
  booktitle={2019 IEEE 58th Conference on Decision and Control (CDC)},
  pages={4373--4378},
  year={2019},
  organization={IEEE}
}

@article{huang2017probabilistic,
  title={Probabilistic safety verification of stochastic hybrid systems using barrier certificates},
  author={Huang, Chao and Chen, Xin and Lin, Wang and Yang, Zhengfeng and Li, Xuandong},
  journal={ACM Transactions on Embedded Computing Systems (TECS)},
  volume={16},
  number={5s},
  pages={1--19},
  year={2017},
  publisher={ACM New York, NY, USA}
}

@inproceedings{abate2021fossil,
  title={FOSSIL: a software tool for the formal synthesis of lyapunov functions and barrier certificates using neural networks},
  author={Abate, Alessandro and Ahmed, Daniele and Edwards, Alec and Giacobbe, Mirco and Peruffo, Andrea},
  booktitle={Proceedings of the 24th international conference on hybrid systems: computation and control},
  pages={1--11},
  year={2021}
}

@inproceedings{feng2020unbounded,
  title={Unbounded-time safety verification of stochastic differential dynamics},
  author={Feng, Shenghua and Chen, Mingshuai and Xue, Bai and Sankaranarayanan, Sriram and Zhan, Naijun},
  booktitle={International Conference on Computer Aided Verification},
  pages={327--348},
  year={2020},
  organization={Springer}
}

@article{lewis2024verification,
  title={Verification of Quantum Circuits through Discrete-Time Barrier Certificates},
  author={Lewis, Marco and Soudjani, Sadegh and Zuliani, Paolo},
  journal={arXiv preprint arXiv:2408.07591},
  year={2024}
}

@inproceedings{sheeran2000checking,
  title={Checking safety properties using induction and a SAT-solver},
  author={Sheeran, Mary and Singh, Satnam and St{\aa}lmarck, Gunnar},
  booktitle={International conference on formal methods in computer-aided design},
  pages={127--144},
  year={2000},
  organization={Springer}
}

@inproceedings{donaldson2011software,
  title={Software verification using k-induction},
  author={Donaldson, Alastair F and Haller, Leopold and Kroening, Daniel and R{\"u}mmer, Philipp},
  booktitle={Static Analysis: 18th International Symposium, SAS 2011, Venice, Italy, September 14-16, 2011. Proceedings 18},
  pages={351--368},
  year={2011},
  organization={Springer}
}

@article{xue2024sufficient,
  title={Sufficient and Necessary Barrier-like Conditions for Safety and Reach-avoid Verification of Stochastic Discrete-time Systems},
  author={Xue, Bai},
  journal={arXiv preprint arXiv:2408.15572},
  year={2024}
}

@article{clark2021control,
  title={Control barrier functions for stochastic systems},
  author={Clark, Andrew},
  journal={Automatica},
  volume={130},
  pages={109688},
  year={2021},
  publisher={Elsevier}
}

@article{xue2024reach,
  title={Reach-avoid controllers synthesis for safety critical systems},
  author={Xue, Bai},
  journal={IEEE Transactions on Automatic Control},
  year={2024},
  publisher={IEEE}
}

@inproceedings{berger2024cone,
  title={Cone-Based Abstract Interpretation for Nonlinear Positive Invariant Synthesis},
  author={Berger, Guillaume and Ghanbarpour, Masoumeh and Sankaranarayanan, Sriram},
  booktitle={Proceedings of the 27th ACM International Conference on Hybrid Systems: Computation and Control},
  pages={1--16},
  year={2024}
}

@book{kushner1967stochastic,
  author={Kushner, Harold Joseph},
  title={Stochastic Stability and Control},
  isbn={9780080955407},
  series={Mathematics in Science and Engineering},
  year={1967},
  publisher={Academic Press}
}

@article{bhattacharjee2024discrete,
  title={A discrete-time dynamical model of prey and stage-structured predator with juvenile hunting incorporating negative effects of prey refuge},
  author={Bhattacharjee, Debasish and Ray, Nabajit and Das, Dipam and Sarmah, Hemanta Kumar},
  journal={Partial Differential Equations in Applied Mathematics},
  volume={10},
  pages={100710},
  year={2024},
  publisher={Elsevier}
}

@inproceedings{abate2025quantitative,
  title={Quantitative supermartingale certificates},
  author={Abate, Alessandro and Giacobbe, Mirco and Roy, Diptarko},
  booktitle={International Conference on Computer Aided Verification},
  pages={3--28},
  year={2025},
  organization={Springer}
}

@inproceedings{henzinger2025supermartingale,
  title={Supermartingale certificates for quantitative omega-regular verification and control},
  author={Henzinger, Thomas A and Mallik, Kaushik and Sadeghi, Pouya and {\v{Z}}ikeli{\'c}, {\DJ}or{\dj}e},
  booktitle={International Conference on Computer Aided Verification},
  pages={29--55},
  year={2025},
  organization={Springer}
}

\appendix
\label{appendix}

\section{Inductive Invariants and Interpolation}
\label{appendix:iii}

We describe the notion of inductive invariants as discussed in \cite{bradley2011sat}. 
Consider a nonstochastic finite-state system, where the state set is a set of logical values while the initial set of states and transition map are described by propositional logical formulae. 
That is, $\Xx \subseteq \{\mathsf{true},\mathsf{false}\}^n$, $\Xx_0 = \{x : I(x) = \mathsf{true} \}$, and $x' = f(x)$  such that $T(x,x') = \mathsf{true}$, where the formula $I(x)$ is a logical formula describing the initial condition over the states of the system $x$, and $T(x,x')$ is a logical formula representing the transition relation from a current state $x$ to a next state $x'$. 
We look at a safety property expressed by a logical formula $P(x)$ described over the state variable $x \in \Xx$. 
We say that such a system satisfies a safety property if, for every reachable state $x \in \Xx$ from the initial set, $P(x) = \mathsf{true}$. 
A prominent effective approach to prove safety is through the use of inductive invariants. 
A formula $Q$ is said to be an inductive invariant, if: 
\begin{itemize}
    \item $\forall x \in \Xx$, we have $I(x) \implies Q(x)$, and
    \item $\forall x, x' \in \Xx$, we have $Q(x) \land T(x,x') \implies Q(x')$. 
\end{itemize}

Observe that any reachable state $x$ satisfies an inductive invariant formula $Q$. 
Thus, showing that a safety property $P$ is an inductive invariant acts as a proof of safety. 
When we fail to prove $P$ to be an inductive invariant (that is $I(x)\notimplies P(x)$ and/or $P(x) \land T(x,x') \notimplies P(x')$), we try to \emph{strengthen} $P$. 
Property $\underline{P}$ is said to be an inductive strengthening of a non-inductive safety property $P$ if there exists a formula $F$ such that $\underline{P} = F \land P$ is inductive. 
Interpolation \cite{mcmillan2003interpolation} is one of these techniques used in the inductive strengthening process.

For interpolation, we unroll the transition relation for some $\ell \in \mathbb{N}$ times and construct a formula representing all possible execution paths from an initial state (assuming all states before the $\ell^{th}$ step are safe). $x_i$ is the state after the $i^{th}$ transition. Then the sequence of states for this unrolling is given by:
\begin{align}
    \label{eq:unrolling}
    I(x_0)\land T(x_0,x_1)\land \dots \land T(x_{\ell-1},x_{\ell})\land \neg P(x_\ell).
\end{align} 
For $\ell = 0$, the formula reduces down to $I(x_0) \land \neg P(x_0)$. 

If formula \eqref{eq:unrolling} is satisfied, then we conclude that the system is unsafe. Otherwise, we use interpolation to try to prove safety by finding intermediate logical formulae called interpolants via Craig's interpolation theorem
\cite{mcmillan2003interpolation} as follows.

\begin{theorem}[Craig's interpolation theorem]
\label{thm:craig_interp}
    Given a pair of clauses (a disjunction of boolean variables or their negation) $E$ and $G$ such that $E \land G$ is unsatisfiable, there exists an intermediate interpolant clause $F$ such that:
    \begin{itemize}
        \item $E \implies F$,
        \item $F \land G$ is unsatisfiable, and
        \item $F$ refers to the common variables of $E$ and $G$.
    \end{itemize}
\end{theorem}

The proof of Theorem \ref{thm:craig_interp} can be found in \cite{troelstra2000basic}. 

Based on this theorem, when formula \eqref{eq:unrolling} is unsatisfiable, there exists intermediate formulae $F_j$ such that formula \eqref{eq:unrolling} can be broken down as follows: 

\begin{align}
    &\underbrace{I(x_0)}_{E_0(x_0)}\land \underbrace{T(x_0,x_1)\land \dots \land T(x_{\ell-1},x_{\ell})\land \neg P(x_{\ell})}_{G_0(x_0,x_1\dots,x_{\ell})} \\ &\hspace{5em}\text{is unsatisfiable, then} 
    \nonumber \\
    &\text{$G_0$ can iteratively be separated as follows:} \nonumber\\
    \label{eq:intrp}
    &\begin{cases} 
      I(x_0) \implies F_0(x_0)\\
      F_0(x_0) \land T(x_0,x_1) \implies F_1(x_1)\\
      \vdots\\
      F_{\ell-1}(x_{\ell-1}) \land T(x_{\ell-1},x_{\ell}) \implies F_{\ell}(x_{\ell}) \text{, and}\\
      F_{\ell}(x_\ell)\land \neg P(x_{\ell}) \text{  is unsatisfiable.}\\
    \end{cases}&
\end{align}

Condition \eqref{eq:intrp} says that the set of $x_j \in \Xx$ where the formula $F_j(x_j)$ is true is an over-approximation of states reachable in $j$ steps and states satisfying $F_j(x_j)$ will not violate the safety property after $(\ell-j)$ transitions. The interpolants can be computed as shown in \cite{bradley2011sat, bradley2012understanding}. We start with $\ell = 0$ and incrementally compute a sequence of interpolants $F_0(x_0) = I(x_0), F_1(x_1), \dots, F_{\ell}(x_{\ell})$ by setting $E(x_{i},x_{i+1}) = F(x_{i})\land T(x_{i},x_{i+1})$ and $G(x_{i+1},\dots,x_{\ell}) = T(x_{i+1},x_{i+2})\land \dots \land T(x_{\ell-1},x_{\ell})$ $\land \neg P(x_{\ell})$ according to Theorem \ref{thm:craig_interp}. 
This iterative process is stopped when the union of the initial formula and all computed interpolants contains all reachable states.
Bringing it to barrier certificates, $F_i(x_i)$ looks like $(0 \leq \Bb_i(x) \leq \gamma \prod_{j=0}^{i-1} \alpha_j)$ and the final inductive invariant that is a proof of safety would be $\cup_{i = 0}^{\ell} \{x: 0 \leq \Bb_i(x) \leq \gamma \prod_{j=0}^{i-1} \alpha_j \}$.
See the next section for intuition of how we reach at these sets.

\section{Intuitions of certificates with diagrams}

Here, we present the intuition behind our proposed certificates, mainly using the nonstochastic setting.

\subsection{IBC}
\label{appendix:ibc}

\begin{figure}[!t]
    \centering
    \epsfig{file=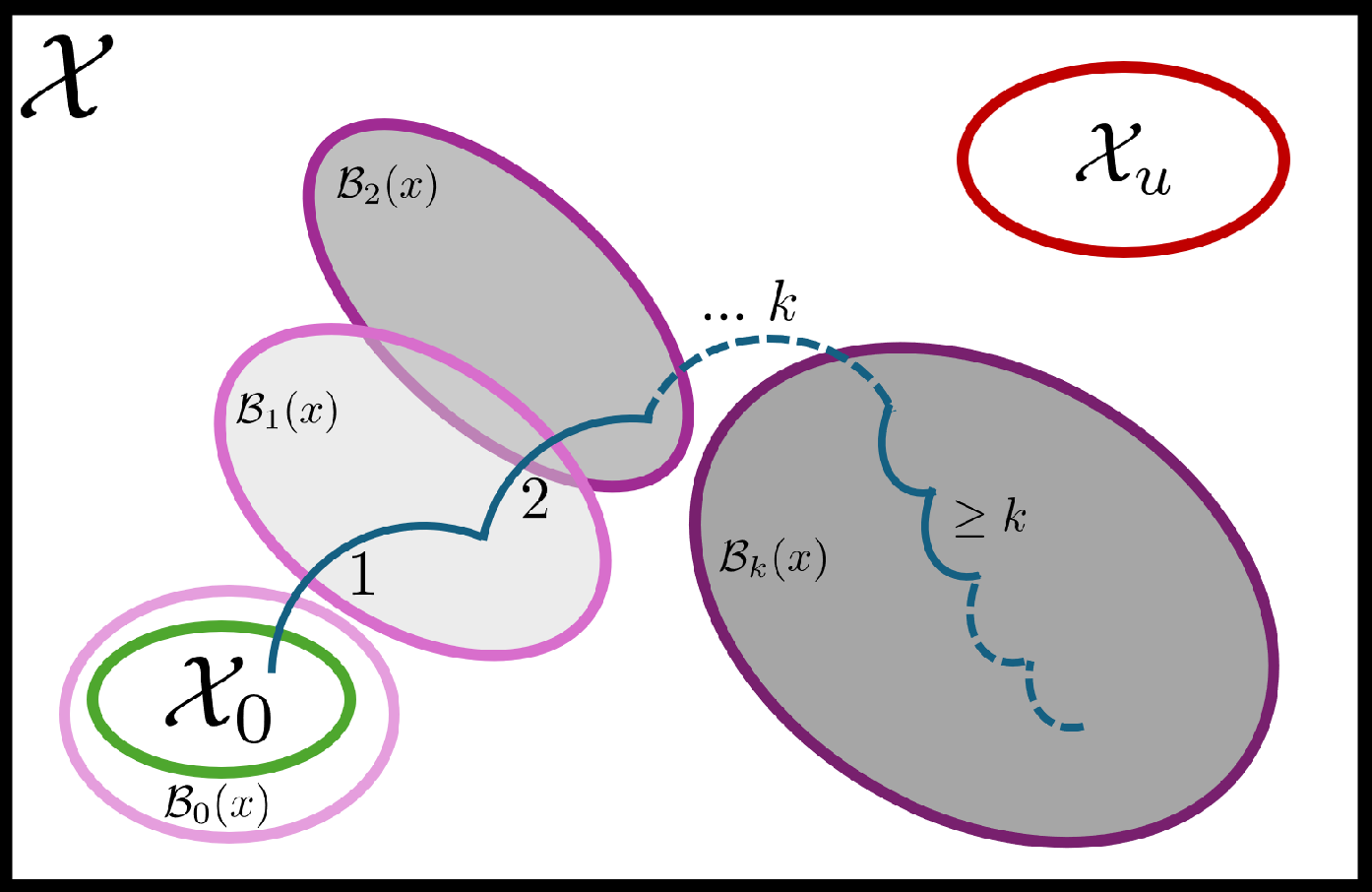, width=0.55\textwidth, keepaspectratio}
    \caption{Diagram of an IBC for a nonstochastic system. As each function serves as an overapproximation, there may be overlaps. The union of all the functions gives the inductive invariant set.} 
    \label{fig:ibc_diagram}
\end{figure}

In the nonstochastic case, conditions~\eqref{eq:ibc3} and~\eqref{eq:ibc4} would look like $\Bb_{i+1}(f(x)) \leq \alpha_i \Bb_i(x)$ and $\Bb_{\ell}(f(x)) \leq \Bb_{\ell}(x)$, respectively. 
Here, in condition~\eqref{eq:ibc3}, $\Bb_{i+1}$ is introduced as part of the (logical) interpolation procedure as introduced in Appendix A. 
For the sake of demonstration, take $\ell = 2$. 
From here, we have $\Bb_1(f(x)) \leq \alpha_0 \Bb_0(x) \leq \alpha_0 \gamma$ and $\Bb_2(f(x)) \leq \alpha_1 \Bb_1(x) \leq \alpha_0 \alpha_1 \gamma$.
As such, the sets $\{x : 0\leq \Bb_1(x) \leq \alpha_0 \gamma\}$ and $\{x : 0\leq \Bb_2(x) \leq \alpha_0 \alpha_1 \gamma\}$ overapproximate sets of states reachable in \emph{one} and \emph{two} steps from the initial set, respectively. 
We additionally have $\Bb_2(f(x)) \leq \Bb_2(x) \leq \alpha_0 \alpha_1 \gamma$ which makes the set $\{x : 0\leq \Bb_2(x) \leq \alpha_0 \alpha_1 \gamma\}$ also overapproximate set of states reachable in \emph{more than two} steps.
Thus, the sets of $\Bb_i$ reflect an overapproximation of the set of states reachable in $i$ steps while the set of $\Bb_{\ell}$ reflect the set of states reachable in $\geq \ell$ steps.
This is how $\ell$, a subscript of the functions, translates to corresponding units of time.
Consequently, $\{x : 0\leq \Bb_0(x) \leq \gamma\} \cup \{x : 0\leq \Bb_1(x) \leq \alpha_0 \gamma\} \cup \{x : 0\leq \Bb_2(x) \leq \alpha_0 \alpha_1 \gamma\}$ overapproximates all the reachable sets of the system. 

In the stochastic setting, we no longer have the advantage of cleanly partitioning states in terms of reachability according to each function’s set, as we did in the deterministic (nonstochastic) case. 
Nonetheless, we can leverage the concept of multiple functions by working with expectations along the system’s evolution and then computing the corresponding safety probability, as encoded in conditions~\eqref{eq:ibc0}-\eqref{eq:ibc4}.

\subsection{$k$-BC and $k$-IBC}
\label{appendix:kbc}

\begin{figure}[!t]
    \centering
    \epsfig{file=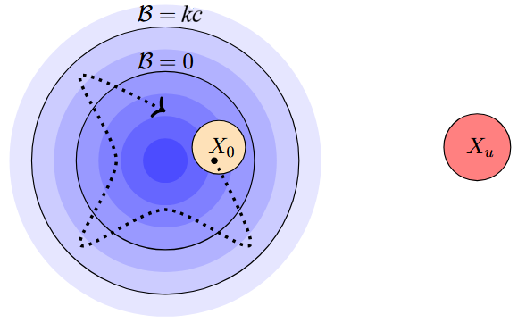, width=0.55\textwidth, keepaspectratio}
    \caption{Diagram of a single-function $k$-BC for a nonstochastic system. Every $k$ steps, the trajectory ends in the zero-sublevel set of the certificate.} 
    \label{fig:kbc_diagram}
\end{figure}

\begin{figure}[!t]
    \centering
    \epsfig{file=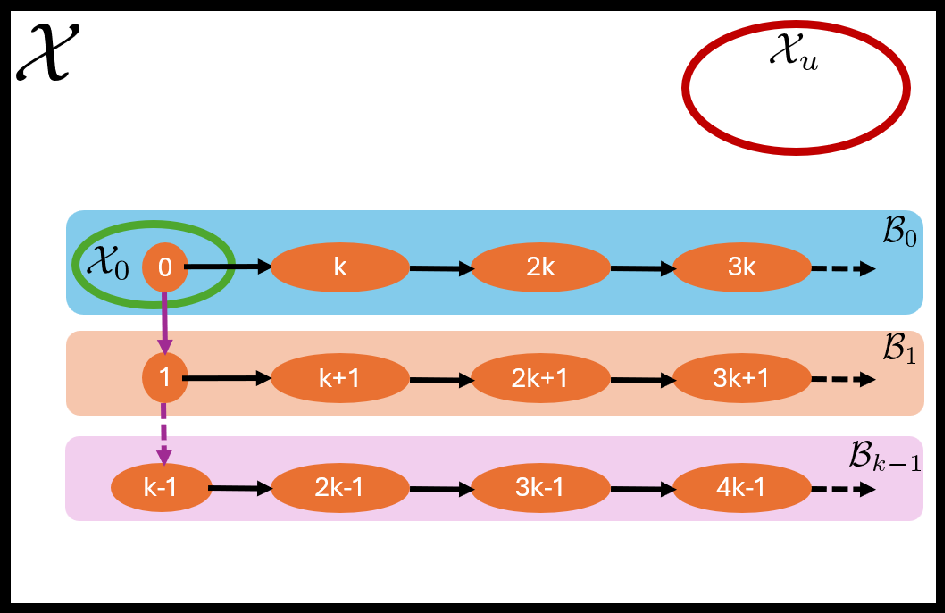, width=0.55\textwidth, keepaspectratio}
    \caption{Diagram of a multi-function $k$-BC for a nonstochastic system. The horizontal black arrows represent $k$ step transitions while the vertical purple arrows are for one step transitions.} 
    \label{fig:kbc_multi_diagram}
\end{figure}

The central idea of a single-function $k$-BC for a nonstochastic system is that the function is permitted to increase by at most a bounded amount $c$ at each step, for up to $k-1$ consecutive steps, but over any sequence of $k$ steps its value must never increase overall. This concept is illustrated in Figure~\ref{fig:kbc_diagram}, adapted from~\cite{anand2021safety}. 

For a multi-function $k$-BC, in the nonstochastic setting, the terms $\E[\Bb_{i}(f^i(x_0,w_i)) | x_0]$ and $\E[\Bb_i(f^k(x,w_k)) | x]$ from conditions~\eqref{eq:kbc3} and~\eqref{eq:kbc4} should be replaced by $\Bb_{i}(f^i(x_0))$ and $\Bb_i(f^k(x))$, respectively.
The multiple functions of a $k$-BC partition time so that the $i^{th}$ function overapproximates the states reachable at time steps of the form $t = nk + i$ for some nonnegative integer $n$.
In this way, all possible trajectories of the system—that is, all reachable states—are still fully captured by the formulation.
Note that if we can construct a single function that serves as a $k$-BC, then we can also construct a collection of functions satisfying the $k$-BC conditions.
However, the reverse implication does not necessarily hold.

For $k$-IBC v1, we retain the IBC condition while relaxing the requirement on $\Bb_\ell$ by altering the supermartingale constraint: we use a $c$-martingale condition for the intermediate transitions and enforce a supermartingale condition only every $k$ steps. In $k$-IBC v2, we introduce a specific scheme that integrates IBC with $k$-BC based on multiple functions.

\end{document}